\documentclass[11pt]{amsart}
\usepackage[top=3cm,left=3cm,right=3cm,bottom=3cm]{geometry}
\usepackage{amsfonts}
\usepackage{mathrsfs}
\usepackage{amsmath,amsthm,amscd,amsfonts,mathrsfs,amssymb,graphicx,enumerate}
\usepackage{xypic}

\newtheorem{thm}{Theorem}[section]
\newtheorem{cor}[thm]{Corollary}
\newtheorem{lem}[thm]{Lemma}
\newtheorem{prop}[thm]{Proposition}
\theoremstyle{definition}
\newtheorem{defn}[thm]{Definition}
\newtheorem{assm}[thm]{Assumption}
\newtheorem{rmk}[thm]{Remark}

\DeclareMathOperator{\ch}{char}

\DeclareMathOperator{\End}{End} 
 \DeclareMathOperator{\Div}{Div}
\DeclareMathOperator{\divisor}{div} 
\DeclareMathOperator{\Proj}{Proj} \DeclareMathOperator{\Sym}{Sym}
\DeclareMathOperator{\Pic}{Pic} \DeclareMathOperator{\Alb}{Alb}

\newcommand{\C}{\ensuremath\mathbb{C}}

\newcommand{\Z}{\ensuremath\mathbb{Z}}
\newcommand{\Q}{\ensuremath\mathbb{Q}}

\newcommand{\fa}{\ensuremath\mathfrak{a}}

\newcommand{\PP}{\ensuremath\mathbb{P}}
\newcommand{\calO}{\ensuremath\mathcal{O}}

\newcommand{\Set}[2]{\left\{#1:#2\right\}}

\begin{document}
\title{On relations among 1-cycles on cubic hypersurfaces}
\author{Mingmin Shen}
\address{DPMMS, University of Cambridge, Wilberforce Road, Cambridge CB3 0WB, UK}
\email{M.Shen@dpmms.cam.ac.uk}


\date{}
\maketitle

\begin{abstract}
In this paper we give two explicit relations among 1-cycles modulo
rational equivalence on a smooth cubic hypersurface $X$. Such a
relation is given in terms of a (pair of) curve(s) and its secant
lines. As the first application, we reprove Paranjape's theorem that
$\mathrm{CH}_1(X)$ is always generated by lines and that it is
isomorphic to $\Z$ if the dimension of $X$ is at least 5. Another
application is to the intermediate jacobian of a cubic threefold
$X$. To be more precise, we show that the intermediate jacobian of
$X$ is naturally isomorphic to the Prym-Tjurin variety constructed
from the curve parameterizing all lines meeting a given rational
curve on $X$. The incidence correspondences play an important role
in this study. We also give a description of the Abel-Jacobi map for
1-cycles in this setting.
\end{abstract}

\section{Introduction}
Let $X\subset\PP^n_k$, where $n\geq 3$, be a smooth cubic
hypersurface in projective space over an algebraically closed field
$k$. In this paper, we are going to study relations among 1-cycles
on $X$. We explain the main theorem (Theorem \ref{main theorem}) of
this paper in the following special situations. Let $C\subset X$ be
a general smooth curve on $X$. Then there are finitely many lines,
$E_i\subset X$, meeting $C$ in two points. These lines will be
called the secant lines of $C$. The first relation we get is
$$
(2e-3)C+\sum E_i =\text{const.}
$$
in $\mathrm{CH}_1(X)$, where $e=\deg(C)$. A second relation is about
a pair of curves on $X$. A simple version goes as follows. Let $C_1$
and $C_2$ be a pair of general smooth curves on $X$. Then there are
finitely many lines, $E_i\subset X$, meeting both $C_1$ and $C_2$.
These are called secant lines of the pair $(C_1,C_2)$. Then our
second relation reads
$$
2e_2C_1+2e_1C_2+\sum E_i=\text{const.}
$$
in $\mathrm{CH}_1(X)$, where $e_1=\deg(C_1)$ and $e_2=\deg(C_2)$. In
all cases, the right hand side is a multiple of the class of the
restriction of a linear $\PP^2$ to $X$. Note that in these relations
the dimension of $X$ does not appear in the coefficients and they
hold in all characteristics. As an application, we reprove the
following theorem of K.H.Paranjape, see \cite{paranjape}(Proposition
4.2).
\begin{thm}
Let $X/k$ be a smooth cubic hypersurface as above. Then the Chow
group $\mathrm{CH}_1(X)$ of 1-cycles on $X$ is generated by lines.
When $\dim X\geq5$, we have $\mathrm{CH}_1(X)\cong\Z$.
\end{thm}

If $X\subset\PP^4_k$ is a smooth cubic threefold with
$\mathrm{char}(k)\neq 2$, then $\mathrm{CH}_1(X)$ is
well-understood. Let $\mathrm{A}_1(X)\subset\mathrm{CH}_1(X)$ be the
subgroup of 1-cycles algebraically equivalent to zero modulo
rational equivalence. Then $\mathrm{A}_1(X)$ is isomorphic to the
intermediate jacobian $J(X)$ of $X$. A precise definition of the
intermediate jacobian will be given in Section 5, Definition
\ref{defn of intermediate jacobian}. Roughly speaking, $J(X)$ is the
universal abelian variety with a principal polarization into which
$\mathrm{A}_1(X)$ maps. One important feature is that for any
algebraic family $Z\rightarrow T$ of 1-cycles on $X$, there is an
induced Abel-Jacobi map $\Psi_T:T\rightarrow J(X)$, which is a
morphism between varieties. Associated to $\Psi_T$, we also have
$\psi_T:\Alb(T)\rightarrow J(X)$, which is a homomorphism of abelian
varieties and is also called the Abel-Jacobi map. There are two ways
to realize $J(X)$ from the geometry of $X$. If we use $F=F(X)$ to
denote the Fano surface of lines on $X$. Then the Albanese variety
$\Alb(F)$ of $F$ carries a natural principal polarization and the
Abel-Jacobi map induces an isomorphism between $\Alb(F)$ and $J(X)$
as principally polarized abelian varieties. A second realization
goes as follows. Let $l\subset X$ be a general line on $X$. Then all
lines meeting $l$ are parameterized by a smooth curve
$\tilde{\Delta}$ of genus 11. The curve $\tilde{\Delta}$ carries a
natural fixed-point-free involution whose quotient is a smooth plane
quintic $\Delta$. Then the associated Prym variety
$\mathrm{Pr}(\tilde{\Delta}/\Delta)$ is naturally isomorphic to
$J(X)$.

As another application of the above natural relations, we give a
third realization of the intermediate jacobian $J(X)$ as a
Prym-Tjurin variety. We refer to \cite{tjurin}, \cite{bloch-murre}
and \cite{kanev} for the basic facts about Prym-Tjurin varieties.
Let $C\subset X$ be a general smooth rational curve. We use
$\tilde{C}$ to denote the normalization of the curve parameterizing
all lines meeting $C$. Then let $D_U=\{([l_1],[l_2])\in
\tilde{C}\times\tilde{C}\}$ such that $l_1$ and $l_2$ are not secant
lines of $C$ and that $l_1$ and $l_2$ meet transversally in a point
not on $C$. We take $D$ to be the closure of $D_U$ in
$\tilde{C}\times\tilde{C}$, which can be called the incidence
correspondence on $\tilde{C}$. Then $D$ defines an endomorphism $i$
of $\tilde{J}=J(\tilde{C})$. Our next result is the following.
\begin{thm}
Let $X/k$ be a smooth cubic threefold and assume that $\ch(k)\neq2$.
Let $C$, $\tilde{C}$, $\tilde{J}$ and $i$ be as
above. Then the following are true.\\
(a) The endomorphism $i$ satisfies the following quadratic relation
$$
(i-1)(i+q-1)=0
$$
where $q=2\deg(C)$.\\
(b) Assume that $\ch(k)\nmid q$. Then the Prym-Tjurin variety
$P=\mathrm{Pr}(\tilde{C},i)=\mathrm{Im}(i-1)\subset\tilde{J}$
carries a natural principal polarization whose theta divisor $\Xi$
satisfies
$$
\Theta_{\tilde{J}}|_P\equiv q\Xi
$$
(c) Assume that $\ch(k)\nmid q$. Then the Abel-Jacobi map
$\psi=\psi_{\tilde{C}}:\tilde{J}\rightarrow J(X)$ factors through
$i-1$ and gives an isomorphism
$$
u_C:(\mathrm{Pr}(\tilde{C},i),\Xi)\rightarrow (J(X),\Theta_{J(X)})
$$
\end{thm}

\begin{rmk}
This is a simplified version of Theorem \ref{J(X) as prym-tjurin}.
When $C$ is a general line on $X$, the above construction recovers
the Prym realization of $J(X)$.
\end{rmk}
We next study how the above construction varies when we change the
curve $C$. Let $C_1$ and $C_2$ be two general smooth rational curves
on $X$. Let
$D_{21}=\{([l_1],[l_2])\in\tilde{C}_1\times\tilde{C}_2\}$ such that
$l_1$ meets $l_2$. We view $D_{21}$ as a homomorphism from
$\tilde{J}_1$ to $\tilde{J}_2$.

\begin{prop}
Assume that $\ch(k)\nmid q_1q_2$. Then the following are true\\
(a) The image of $D_{21}$ is $\mathrm{Pr}(\tilde{C}_2,i_2)$. The
homomorphism $D_{21}$ factors through $i_1-1$ and gives an
isomorphism
$$
t_{21}:(\mathrm{Pr}(\tilde{C}_1,i_1),\Xi_1)\rightarrow
(\mathrm{Pr}(\tilde{C}_2,i_2),\Xi_2)
$$
(b) The isomorphism $t_{21}$ is compatible with $u_{C_i}$, namely
$u_{C_2}\circ t_{21}=u_{C_1}$.
\end{prop}

We also get a description for the Abel-Jacobi map for a family of
1-cycles. Fix a general smooth rational curve $C\subset X$ such that
$\ch(k)\nmid q$. Let $\mathscr{C}\subset X\times T$ be a family of
curves on $X$ parameterized by $T$. Assume that for general $t\in
T$, there are finitely many secant lines $L_i$ of the pair
$(C,\mathscr{C}_t)$. Then the rule $ t\mapsto \sum [L_i]$ defines a
morphism
$$
\Psi_{T,\tilde{J}}:T\rightarrow \tilde{J}
$$
\begin{prop}
Let notations and assumptions be as above. Then the image of
$\Psi_{T,\tilde{J}}$ is contained in
$\mathrm{Pr}(\tilde{C},i)\subset\tilde{J}$. The composition
$u_C\circ\Psi_{T,\tilde{J}}$ is identified with the Abel-Jacobi map
associated to $\mathscr{C}/T$.
\end{prop}

We summarize the structure of the paper. In section 2, we review the
theory of secant bundles on symmetric products. We state the results
in the form that we are going to use and all proofs are included for
completeness. Section 3 is devoted to the construction and basic
properties of residue surfaces associated to a curve or a pair of
curves. Those surfaces play an important role in later proofs in
section 4. In section 4, we state and prove our main theorem on
relations among 1-cycles. We also derive Paranjape's theorem as a
corollary. Section 5 is a review on basic results on Fano scheme of
lines and the intermediate jacobian of a cubic threefold. In section
6, we show how one can realize the intermediate jacobian of a cubic
threefold
as a Prym-Tjurin variety described above.\\

\indent \textit{Acknowledgement}. The author would like to thank C.
Vial for many interesting discussions, B. Totaro  for reading the
manuscript and R. Friedman for informing the author of many useful
references. The author thanks the referee, whose suggestions
significantly improved the exposition of this paper. This work was
carried out during the author's first year of Simons Postdoctoral
Fellowship. He would like to thank Simons
Foundation for this support.\\

\textbf{Notations and conventions:} \\
1. $\mathrm{G}(r,V)$ is the Grassmannian rank $r$ quotient of $V$;
$\mathscr{E}_r$ is the canonical rank $r$ quotient bundle of $V$ on
$\mathrm{G}(r,V)$;\\
2. $\PP(V)=\Proj(\Sym^*V)=\mathrm{G}(1,V)$;\\
3. $\mathrm{G}(r_1,r_2,V)$ is the flag variety of successive
quotients of $V$ with ranks $r_2>r_1$;\\
4. $N(e,g)=\frac{5e(e-3)}{2}+6-6g$;\\
5. We use $\equiv$ to denote numerical equivalence; \\
6. A general curve on $X$ means that it comes from a dense open
subset of the corresponding component of the Hilbert scheme of curves on $X$; \\
7. For two algebraic cycles $D_1$ and $D_2$, an equation $D_1=D_2$
can mean equality either as cycles or modulo rational equivalence;\\
8. For $P,Q\in\PP(V)$ distinct, we use $\overline{PQ}$ to denote the
line through both $P$ and $Q$.

\section{Secant bundles on symmetric products}
Most of the results in this section are special cases of that of
\cite{matt}. We state the results in the form we need and include
the proofs for completeness. Let $C$ be a smooth projective curve
over an algebraically closed field $k$. We use $C^{(2)}$ to denote
the symmetric product of $C$. Let $\pi:C\times C\rightarrow C^{(2)}$
be the canonical double cover which ramifies along the diagonal
$\Delta_C\subset C\times C$. For any invertible sheave $\mathscr{L}$
on $C$, we define the \textit{symmetrization}
$\mathscr{E}(\mathscr{L})$ of $\mathscr{L}$ to be
$\pi_*p_1^*\mathscr{L}$ on $C^{(2)}$, where $p_1:C\times
C\rightarrow C$ is the projection to the first factor. For example,
if $\mathscr{L}=\omega_C$ then
$\mathscr{E}(\mathscr{L})\cong\Omega_{C^{(2)}}^1$.

To compute the Chern classes of $\mathscr{E}(\mathscr{L})$, we
describe several special cycles on $C^{(2)}$. Let $x\in C$ be a
closed point. We define
\begin{equation}
D_x=\pi_*(\{x\}\times C)\subset C^{(2)}
\end{equation}
If $\mathfrak{a}=\sum x_i$ is an effective divisor on $C$ with
$x_i\neq x_j$ for different $i$ and $j$, we write
\begin{equation}\label{first notations}
D_{\mathfrak{a}}=\sum D_{x_i}, \quad \mathfrak{a}^{[2]}=\sum_{i<j}
\pi_*(x_i,x_j),\quad \delta_*\mathfrak{a}=\sum \pi_*(x_i,x_i)
\end{equation}
Here $\delta=\pi\circ \Delta_C:C\rightarrow C^{(2)}$. Note that the
definition of $D_\mathfrak{a}$ can be extended linearly to all
divisors. The projection formula implies
\begin{align*}
D_x\cdot D_y & =\pi_*(x\times C \cdot \pi^*\pi_*(y\times C))\\
&=\pi_*(x\times C\cdot (y\times C+C\times y))\\
&=\pi_*(x,y)
\end{align*}
for any $x,y\in C$. Then we easily get
\begin{equation}
(D_\mathfrak{a})^2=\delta_*\mathfrak{a}+2\,\mathfrak{a}^{[2]}
\end{equation}
Let $\Delta_0=c_1(\pi_*\calO_{C\times C})$ in
$\mathrm{CH}_1(C^{(2)})$.

Assume that $\mathscr{L}$ is very ample and for some
$V\subset\mathrm{H}^0(C,\mathscr{L})$, it defines a closed immersion
$i:C\rightarrow \PP(V)$. Note that any two points $x$ and $y$ on $C$
define a line $\overline{i(x)\,i(y)}$ (the tangent line at $i(x)$ if
$x=y$). This defines a natural morphism $\varphi:C^{(2)}\rightarrow
\mathrm{G}(2,V)$. Let
$V\otimes\calO_{\mathrm{G}(2,V)}\twoheadrightarrow\mathscr{E}_2$ be
the universal rank 2 quotient bundle.
\begin{prop}
We have an isomorphism $\mathscr{E}(\mathscr{L})\cong
\varphi^*\mathscr{E}_2$.
\end{prop}
\begin{proof}
If we view $\mathrm{G}(2,V)$ as the space on lines in $\PP(V)$, then
by pulling back the universal family of lines, we get the following
diagram
$$
\xymatrix{
  \PP(\varphi^*\mathscr{E}_2)\ar[d]_{\tilde{\pi}}\ar[r]^f &\PP(V)\\
  C^{(2)} &
}
$$
and we have $\varphi^*\mathscr{E}_2\cong
\tilde{\pi}_*f^*\calO_{\PP(V)}(1)$.  The fiber $\tilde\pi^{-1}(t)$
is the line $\varphi(t)\in\mathrm{G}(2,V)$ and $f$ is the inclusion
of the line into $\PP(V)$. Note that there are two distinguished
points $x$ and $y$ on the fiber
$\tilde{\pi}^{-1}(t)=\overline{i(x)\,i(y)}$, where $t\in C^{(2)}$ is
the point $\pi(x,y)$. Those distinguished points form a divisor $D$
on $\PP(\varphi^*(\mathscr{E}_2))$. By choosing an isomorphism
$D\cong C\times C$, we can assume that $f|_D=i\circ p_1$ and
$\tilde{\pi}|_D=\pi$. Then $f^*\calO_{\PP(V)}(1)|_D
=p_1^*\,i^*\calO_{\PP(V)}(1)=p_1^*\mathscr{L}$. Consider the
following short exact sequence
$$
\xymatrix@C=0.5cm{
  0 \ar[r] & f^*\calO_{\PP(V)}(1)\otimes\calO_{\PP(\varphi^*\mathscr{E}_2)}(-D) \ar[rr] &&f^*\calO_{\PP(V)}(1) \ar[rr] && f^*\calO_{\PP(V)}(1)|_D \ar[r] & 0 }
$$
Apply $R\pi_*$ and note that the left term restricts to
$\calO_{\PP^1}(-1)$ on each fiber of $\tilde{\pi}$, we get $
\varphi^*\mathscr{E}_2=\tilde{\pi}_*f^*\calO_{\PP(V)}(1)\cong
\pi_*p_1^*\mathscr{L}=\mathscr{E}(\mathscr{L}). $
\end{proof}

\begin{cor}
We have the equality
$c_2(\mathscr{E}(\mathscr{L}))=\mathfrak{a}^{[2]}$ in
$\mathrm{CH}_0(C^{(2)})$, where $\mathfrak{a}$ is a general element
of the complete linear system $|\mathscr{L}|$.
\end{cor}
\begin{proof}
It is known from Schubert calculus (see \cite{ful} 14.7) that
$c_2(\mathscr{E}_2)$ is represented by the cycle defined by the
space of all lines that are contained in a hyperplane
$H\subset\PP(V)$. When $H$ is chosen to be general, then
$\mathfrak{a}=i^*H$ is an element of $|\mathscr{L}|$ such that
$\mathfrak{a}^{[2]}$ is defined. Then we have $
c_2(\mathscr{E}(\mathscr{L}))=c_2(\varphi^*\mathscr{E}_2)
=\varphi^*c_2(\mathscr{E}_2)=\mathfrak{a}^{[2]}$.
\end{proof}

For very ample $\mathscr{L}$, we take a general section
$s\in\mathrm{H}^0(C,\mathscr{L})$ write
$\mathfrak{a}=\divisor(s)=\sum_{i=1}^{d} x_i$, where
$d=\deg\mathscr{L}$. The section $s$ defines a short exact sequence
$$
 \xymatrix@C=0.5cm{
   0 \ar[r] & \calO_C \ar[rr]^{s} && \mathscr{L} \ar[rr] && \oplus_{i=1}^{d}\kappa(x_i) \ar[r] & 0 }
$$
Pull the sequence back to $C\times C$ via $p_1^*$ and then apply
$R\pi_*$, we get
$$
 \xymatrix@C=0.5cm{
   0 \ar[r] & \pi_*(\calO_{C\times C}) \ar[rr] && \mathscr{E}(\mathscr{L}) \ar[rr] && \oplus_{i=1}^{d}\calO_{D_{x_i}} \ar[r] & 0 }
$$
Since $\pi_*\calO_{C\times C}$ admits a nowhere vanishing section
``1", the second Chern class of $\pi_*\calO_{C\times C}$ is zero.
Hence we have $c(\pi_*\calO_{C\times C})=1+\Delta_0$. Take total
Chern classes in the above exact sequence and note that
$c(\calO_{D_i})=\frac{1}{1-D_i}$, we get
\begin{align*}
c(\mathscr{E}(\mathscr{L}))
&=(1+\Delta_0)\prod_{i=1}^{d}(\dfrac{1}{1-D_{x_i}})\\
&=(1+\Delta_0)\prod_{i=1}^{d}(1+D_{x_i}+D_{x_i}^2)\\
&=(1+\Delta_0)(1+\sum_{i=1}^{d}D_{x_i}+\sum_{i=1}^{d}D_{x_i}^2
+\sum_{1\leq i<j\leq d} D_{x_i}\cdot D_{x_j})\\
&=1 + (D_\mathfrak{a}+\Delta_0) + (\Delta_0\cdot
D_\mathfrak{a}+\delta_*{\mathfrak{a}} +\mathfrak{a}^{[2]})
\end{align*}
Hence we get
\begin{equation}\label{ample}
c_1(\mathscr{E}(\mathscr{L}))=D_\mathfrak{a}+\Delta_0,\quad
c_2(\mathscr{E}(\mathscr{L}))=\mathfrak{a}^{[2]}+\Delta_0\cdot
D_\mathfrak{a} +\delta_*\mathfrak{a}
\end{equation}
Compare this with the above Corollary, we have the following
\begin{equation}
\Delta_0\cdot D_{\mathfrak{a}}=-\delta_*\mathfrak{a}
\end{equation}

Now let $\mathscr{L}_1$ and $\mathscr{L}_2$ be two very ample line
bundles on $C$ with degrees $d_1$ and $d_2$ respectively. Pick a
general section $s_i\in\mathrm{H}^0(C,\mathscr{L}_i)$ where $i=1,2$.
Set
$$
\mathfrak{a}_1=\divisor(s_1)=\sum_{i=1}^{d_1}x_i, \quad
\mathfrak{a}_2=\divisor(s_2)=\sum_{i=1}^{d_2}y_i
$$
As before, we can use $s_2$ to construct the following short exact
sequence
$$
\xymatrix@C=0.5cm{
  0 \ar[r] & \mathscr{L}_1\otimes\mathscr{L}_2^{-1} \ar[rr]^{\quad s_2} && \mathscr{L}_1 \ar[rr] && \oplus_{i=1}^{d_2}\kappa(y_i) \ar[r] & 0 }
$$
Similarly, after applying $R\pi_*$, we get a short exact sequence on
$C^{(2)}$,
$$
\xymatrix@C=0.5cm{
  0 \ar[r] & \mathscr{E}(\mathscr{L}_1\otimes\mathscr{L}_2^{-1}) \ar[rr] && \mathscr{E}(\mathscr{L}_1) \ar[rr]
  && \oplus_{i=1}^{d_2}\calO_{D_{y_i}} \ar[r] & 0 }
$$
One easily computes the Chern classes as follows
\begin{align*}
c(\mathscr{E}(\mathscr{L}_1\otimes\mathscr{L}_2^{-1}))
&=(1+D_{\mathfrak{a}_1}+\Delta_0 +
\mathfrak{a}_1^{[2]})\prod_{i=1}^{d_2}(1-D_{y_i})\\
 &=(1+D_{\mathfrak{a}_1}+\Delta_0
 +\mathfrak{a}_1^{[2]})(1-D_{\mathfrak{a}_2}+\mathfrak{a}_2^{[2]})\\
 &=1 + (\Delta_0+D_{\mathfrak{a}_1}-D_{\mathfrak{a}_2}) +
 (\mathfrak{a}_1^{[2]}+\mathfrak{a}_2^{[2]} -\Delta_0\cdot D_{\mathfrak{a}_2} -D_{\mathfrak{a}_1}\cdot D_{\mathfrak{a}_2})
\end{align*}

We summarize the above discussion into the following
\begin{thm}\label{secant bundle}
Let $C$ be a smooth projective curve and $\mathscr{L}$ be a
invertible sheaf. Let $\mathfrak{a}_1$, $\mathfrak{a}_2$ and
$\mathfrak{a}$ be effective divisors on $C$ such that each is a sum
of distinct points with multiplicity 1. Let
$\mathscr{E}=\mathscr{E}(\mathscr{L})$ be the symmetrization of
$\mathscr{L}$. Then the following are true\\
\indent(a) If $\mathscr{L}\cong\calO_C(\mathfrak{a}_1 -
\mathfrak{a}_2)$, then the following equalities hold in the Chow
ring of $C^{(2)}$,
\begin{align}
 c_1(\mathscr{E}) & =\Delta_0+ D_{\mathfrak{a}_1} -
 D_{\mathfrak{a}_2}\\
 c_2(\mathscr{E}) & =\mathfrak{a}_1^{[2]} +\mathfrak{a}_2^{[2]} +
 \delta_*\mathfrak{a}_2 - D_{\mathfrak{a}_1}\cdot D_{\mathfrak{a}_2}
\end{align}
\indent(b) The following identities hold in the Chow ring of
$C^{(2)}$,
\begin{align}
D_{\mathfrak{a}}\cdot D_{\mathfrak{a}} &= \delta_*\mathfrak{a}
+2\,\mathfrak{a}^{[2]},\\
D_{\mathfrak{a}}\cdot\Delta_0 & = -\delta_*\mathfrak{a}\\
2\,\Delta_0 & =-\pi_*(\Delta_C)\\
2\Delta_0\cdot\Delta_0 & = -\delta_*{K_C}
\end{align}
\end{thm}

\begin{proof}
Only the last two equalities need to be proved. By construction we
have the following exact sequence
$$
\xymatrix@C=0.5cm{
  0 \ar[r] & \calO_{C\times C} \ar[rr] && \pi^*\pi_*\calO_{C\times C} \ar[rr] && \calO_{C\times C}(-\Delta_C) \ar[r] & 0 }
$$
This implies that $\pi^*\Delta_0 = c_1(\pi^*\pi_*\calO_{C\times
C})=-\Delta_C$. Hence $-\pi_*\Delta_C=\pi_*\pi^*\Delta_0=2\Delta_0$
and by projection formula we have
$2\Delta_0\cdot\Delta_0=-\pi_*\Delta_C\cdot\Delta_0=\pi_*(\Delta_C\cdot\Delta_C)=-\delta_*K_C$.
\end{proof}

\section{Residue surfaces}
Let $X\subset\PP^n$, $n\geq 4$, be a smooth cubic hypersurface over
an algebraically field $k$. In this section we will construct so
called residue surfaces associated to a curve or a pair of curves on
$X$. Let $V=\mathrm{H}^0(X,\calO_X(1))$.

Let $C\subset X$ be a complete at worst nodal curve on $X$. Note
that $C$ might be disconnected. If $x\in C$ is a nodal point, let
$\Pi_x$ be the plane spanned by the tangent directions of the two
branches at $x$.

\begin{defn}\label{defn of secant line}
A line $l$ on $X$ is called a \textit{secant line} of $C$ if, (first
type) it meets $C$ in two smooth points (which might be
infinitesimally close to each other); or (second type) it passes
through a node $x$ and lies in the plane $\Pi_x$ and also it is not
a component of $C$; or (third type) it is a component of $C$ for
some node $x\in l$ we have $\Pi_x\cdot X=2l+l'$ for some line $l'$.
\end{defn}

\subsection{Residue surface associated to a single curve}
Throughout this subsection, we make the following assumption.
\begin{assm}
The curve $C\subset X$ is smooth irreducible of degree $e\geq 2$ and
genus $g$. Furthermore, $C$ has only finitely many secant lines.
\end{assm}

We want to produce a surface on $X$ which is birational to the
symmetric product of $C$. First we fix several notations. Set
$\mathscr{L}=\calO_X(1)|_C$ and $\Sigma=C^{(2)}$. As in the previous
section, we have a morphism $\varphi:\Sigma\rightarrow
\mathrm{G}(2,V)$ with $\mathscr{E}\cong\varphi^*\mathscr{E}_2$,
where $\mathscr{E}=\mathscr{E}(\mathscr{L})$ is the symmetrization
of $\mathscr{L}$ and $\mathscr{E}_2$ is the canonical rank 2
quotient bundle on $\mathrm{G}(2,V)$. Set $P=\PP(\mathscr{E})$, then
we have the following diagram
$$
\xymatrix{
 D\ar@{^(->}[r] &P\ar[r]^{f\quad}\ar[d]_{\tilde{\pi}} &\PP(V)\\
 &\Sigma &
}
$$
such that a fiber of $\tilde\pi$ maps to a line on $\PP(V)$ and
$D\cong C\times C$ with $f|_D=p_1$ and $\pi=\tilde\pi|_D$ being the
canonical morphism from $C\times C$ to $C^{(2)}$. See the previous
section for more details. Let $\mathfrak{a}\in |\mathscr{L}|$ be a
general hyperplane section, and $\Delta_0=c_1(\pi_*\calO_{C\times
C})$ as before. Set $\xi=c_1(f^*\calO(1))$. By Theorem \ref{secant
bundle}, it is easy to get the following
\begin{equation}\label{relative canonical}
 K_{P/\Sigma}=c_1(\omega_{\tilde\pi})=\tilde\pi^*c_1(\mathscr{E})-2\xi
 =\tilde\pi^*(D_{\mathfrak{a}}+\Delta_0) - 2\xi
\end{equation}
Apply $R\tilde\pi_*$ to the following exact sequence
$$
\xymatrix@C=0.5cm{
  0 \ar[r] & \calO_P(-D) \ar[rr] && \calO_P \ar[rr] && \calO_D \ar[r] & 0 }
$$
and we get the following short exact sequence
$$
\xymatrix@C=0.5cm{
  0 \ar[r] & \tilde\pi_*\calO_P=\calO_\Sigma \ar[rr] && \pi_*\calO_D \ar[rr] && R^1\tilde\pi_*\calO_P(-D) \ar[r] & 0 }
$$
By relative duality, we have
$$
R^1\tilde\pi_*\calO_P(-D) \cong (\tilde\pi_*\calO(D+K_{P/\Sigma}))^*
\cong (\tilde\pi_*\calO(D-2\xi+\tilde\pi^*D_{\mathfrak{a}} +
\tilde\pi^*\Delta_0))^*
$$
Hence by taking the first Chern classes we get
$$
\Delta_0=c_1(\pi_*\calO_D)=c_1(R^1\tilde\pi_*\calO_P(-D))
=-c_1(\tilde\pi_*\calO(D-2\xi+\tilde\pi^*D_{\mathfrak{a}} +
\tilde\pi^*\Delta_0))
$$
Note that $D-2\xi$ is the pull back of some class from $\Sigma$, say
$D-2\xi=\tilde\pi^*D'$. Then the projection formula implies
$$
\tilde\pi_*\calO(D-2\xi+\tilde\pi^*D_{\fa}+\tilde\pi^*\Delta_0)\cong\calO(D'+D_\fa+\Delta_0)
\otimes\tilde\pi_*\calO_P=\calO(D'+D_\fa+\Delta_0)
$$
Combine the above two identities, we have
\begin{equation}\label{class of D}
D=2\xi -\tilde\pi^*(D_{\mathfrak{a}}+2\Delta_0)
\end{equation}
Let $s\in \mathrm{H}^0(\PP(V),\calO_{\PP(V)}(3))=\Sym^3V$ be the
degree 3 homogeneous polynomial whose zero defines
$X\subset\PP(V)=\PP^4$. Then
$$
\divisor(f^*s)=D+S
$$
for some surface $S$ in $P$.
\begin{defn}\label{residue surface single}
The surface $S$ together with the morphism $\phi=f|_S:S\rightarrow
X$ is called the \textit{residue surface} associated to $C\subset
X$.
\end{defn}
\noindent Note that the class of $D+S$ is $3\,\xi$, we combine this
with \eqref{class of D} and easily get the following
\begin{equation}\label{class of S}
S=\xi + \tilde\pi^*(D_{\mathfrak{a}}+2\Delta_0)
\end{equation}
The following lemma was proved in \cite{hrs} (Lemma 4.2). Here we
give a different proof.
\begin{lem}
Counting with multiplicities, there are
$N(e,g)=\frac{5e(e-3)}{2}+6-6g$ secant lines of $C$.
\end{lem}
\begin{proof}
Let $\sigma=\tilde\pi|_S:S\rightarrow \Sigma$. As a divisor on $P$,
the surface $S$ defines a invertible sheaf $\calO_P(S)$ together
with a section $s_0$. Then $s_0$ can be viewed as a section of
$\tilde\pi_*\calO_P(S)$ and the zero locus of this section is
exactly the scheme $B_C$ of lines on $X$ that meet $C$ in two
points. To compute the length of $B_C$, it suffices to compute
$c_2(\tilde\pi_*\calO_P(S))$. From \eqref{class of S}, we know that
$\tilde\pi_*\calO_P(S)=\mathscr{E}\otimes\calO_\Sigma(D_\mathfrak{a}
+2\Delta_0)$. It follows that
\begin{align*}
c_2(\tilde\pi_*\calO_P(S)) & = c_1(\mathscr{E})\cdot (D_\mathfrak{a}
+ 2\Delta_0) + c_2(\mathscr{E}) + (D_\mathfrak{a}+2\Delta_0)^2\\
 & = (D_\mathfrak{a}+\Delta_0)(D_\mathfrak{a}+2\Delta_0)
 +\mathfrak{a}^{[2]} + (D_\mathfrak{a}+2\Delta_0)^2\\
 & = 2 D_\mathfrak{a}^2 + 7D_\mathfrak{a}\cdot\Delta_0+ 6\Delta_0^2
 +\mathfrak{a}^{[2]}\\
 & = 5\mathfrak{a}^{[2]} -5\delta_*\mathfrak{a} -3\delta_*K_C\\
 &\equiv 5e(e-1)/2 -5e +6(1-g)=N(e,g)
\end{align*}
Here $\equiv$ denotes numerical equivalence.
\end{proof}

\begin{defn}\label{well positioned single}
We define the \textit{multiplicity}, $e(C;l)$, of a secant line $l$
to be the length of $B_C$ at the point $[l]$. The curve $C\subset X$
is \textit{well-positioned} if it has exactly $N(e,g)$ distinct
secant lines or equivalently all secant lines are of multiplicity 1.
\end{defn}

\begin{prop}\label{prop of single residue}
Assume that $C$ is neither a line nor a conic. Then the following are true\\
(1) The surface $S$ is smooth if and only if $C$ is well-positioned.
If $C$ is well-positioned, then $S$ is the blow-up of $\Sigma$ at
$N(e,g)$ points.\\
(2) Assume that $C$ is well-positioned, then the following
equalities hold in $\mathrm{CH}_1(S)$,
\begin{align}
\xi|_S &= 2\,\sigma^*D_\mathfrak{a} + 3\,\sigma^*\Delta_0 -
\sum_{i=1}^{N(e,g)}E_i\\
D|_S & = 3\,\sigma^*D_\mathfrak{a} + 4\,\sigma^*\Delta_0 -
2\sum_{i=1}^{N(e,g)}E_i\\
S|_S & =
3\,\sigma^*D_\mathfrak{a}+5\,\sigma^*\Delta_0-\sum_{i=1}^{N(e,g)}E_i
\end{align}
where $\sigma:S\rightarrow \Sigma$ is the blow up, and $E_i$ are the
exceptional curves.
\end{prop}

\begin{proof}
(1) Since $\sigma=\tilde\pi|_S:S\rightarrow \Sigma$ is isomorphism
away from the exceptional locus $E_i$, and each $E_i$ corresponds to
a line on $X$ that meets $C$ in two points. Let $E$ be one of those
$E_i$'s which meets $C$ in points $P_1$ and $P_2$. Let $(x,y)$ be a
set local parameters of $\Sigma$ at $\sigma(E)$. We use $[T_0:T_1]$
to denote the homogeneous coordinates of $E$ with $P_1=[0:1]$ and
$P_2=[\lambda:1]$. Then we have
$$
f^*s=Q(T_0,T_1)\cdot(u(x,y)T_0+v(x,y)T_1)
$$
in a neighborhood of $E$ where $Q(T_0,T_1)$ is a quadratic form in
$T_0$ and $T_1$ with
$$
Q|_{(x,y)=(0,0)}=T_0(T_0-\lambda T_1).
$$
Hence the local equation for $S$ is $uT_0+vT_1=0$. Thus $S$ is
smooth along $E$ if and only if $(u,v)$ generate the maximal idea of
$\sigma(E)$ in $\Sigma$. The last condition is the same as
$\tilde\pi_*(f^*s)$ has a simple zero at $\sigma(E)$ as a section of
$\tilde\pi_*(\calO_P(S))$. Hence $S$ is smooth if and only if $C$ is
well-positioned.

(2) By adjunction formula, we have
$$
K_S=[K_P+S]|_S=(\tilde\pi^*K_{\Sigma}+K_{P/\Sigma}+[S])|_S=\sigma^*K_\Sigma
+\sigma^*(2\,D_\mathfrak{a}+3\,\Delta_0)-\xi|_S
$$
On the other hand $S$ is the blow up of $\Sigma$ and we have
$$
K_S=\sigma^*K_\Sigma +\sum E_i
$$
Compare the above two identities, we get
$$
\xi|_S=\sigma^*(2\,D_\mathfrak{a}+3\,\Delta_0) - \sum E_i
$$
Hence $D|_S$ and $S|_S$ can be computed easily using \eqref{class of
D} and \eqref{class of S}.
\end{proof}

\subsection{Residue surface associated to a pair of curves}
In this section we fix two at worst nodal complete curves
$i_1:C_1\hookrightarrow X$ of degree $e_1$ and genus $g_1$ and
$i_2:C_2\hookrightarrow X$ of degree $e_2$ and genus $g_2$. Let
$C=C_1\cup C_2\subset X$ be the union. We allow $C_1$ and $C_2$ to
meet each other transversally at $x_1,\ldots,x_r$ which are smooth
points on each curve. If $C_1$ and $C_2$ do meet, we use
$\Pi_i=\Pi_{x_i}$ to denote the plane spanned by the tangent
directions of $C_1$ and $C_2$ at the point $x_i$.
\begin{defn}\label{secant line pair}
A line $l\subset X$ is called a \textit{secant line} of the pair
$(C_1,C_2)$ if $l$ is a secant line of $C_1\cup C_2$ but not of
$C_1$ nor of $C_2$.
\end{defn}
We fix the following assumptions for the remaining of this
subsection.
\begin{assm}
Both $C_1$ and $C_2$ are smooth irreducible. There are only finitely
many secant lines of the pair $(C_1,C_2)$. If one of the curves is a
line, then they do not meet each other.
\end{assm}

Let $\mathscr{L}_1=\calO_X(1)|_{C_1}$ and
$\mathscr{L}_2=\calO_X(1)|_{C_2}$. By sending $(x,y)\in C_1\times
C_2$ to the line $\overline{i_1(x)i_2(y)}\subset\PP(V)$, we get a
rational map $ \varphi_0:C_1\times C_2 \dashrightarrow
\mathrm{G}(2,V)$. The point $x_i\in C_1\cap C_2$ determines
$\bar{x}_i\in C_1\times C_2$, $i=1,\ldots,r$, which form the locus
where $\varphi_0$ is not defined. Let $\Sigma$ be the blow up of
$C_1\times C_2$ at the points $\bar{x}_i$ and $F_i$ be the
corresponding exceptional curves, $i=1,\ldots,r$. Then $\varphi_0$
extends to a morphism
$$
\varphi:\Sigma\rightarrow \mathrm{G}(2,V)
$$
such that $\varphi|_{F_i}$ parameterizes all lines through $x_i$
lying on the plane $\Pi_i$. This can be easily seen from the local
description of the map $\varphi_0$. Let $t_\alpha$ be the local
parameter of $C_\alpha$ at $x_i$, $\alpha=1,2$. In homogeneous
coordinates, the line $\overline{i_1(t_1)i_2(t_2)}$ can be
parameterized by $(1-\lambda)u_1(t_1)+\lambda u_2(t_2)$, where
$\lambda$ is a local coordinate of the line. Note that
$u_1(0)=u_2(0)=\vec{a}$ as vectors of $V^*$. Hence
$u_1=\vec{a}+t_1v_1(t_1)$ and $u_2=\vec{a}+t_2v_2(t_2)$. Hence the
above parametrization can be written as
$$
u_1 + \lambda(t_1v_1(t_1) - t_2v_2(t_2))
$$
A neighborhood of $F_i$ has two charts. On the chart with
coordinates $(x=t_1,y=t_2/t_1)$ we define $\varphi(x,y)$ to be the
line parameterized by
$$
u_1 + \lambda(v_1(x) - yv_2(xy))
$$
and on the chart with coordinates $(x'=t_1/t_2,y'=t_2)$, we set
$\varphi(x',y')$ to be the line parameterized by
$$
u_1 + \lambda(x'v_1(x'y') - v_2(y'))
$$
This gives the morphism $\varphi$ as an extension of $\varphi_0$.

Let $\mathscr{E}=\varphi^*\mathscr{E}_2$. The definition of
$\varphi$ implies that on $\Sigma-\cup F_i$ the sheaf $\mathscr{E}$
is isomorphic to $p_1^*\mathscr{L}_1\oplus p_2^*\mathscr{L}_2$. We
also note that $\varphi(F_i)$ is a line (with respect to the
Pl\"ucker embedding) and hence $c_1(\mathscr{E})\cdot F_i=1$.
Combine the above two facts, we have
\begin{equation}
c_1(\mathscr{E})  = \mathfrak{a}_1\times C_2 +
C_1\times\mathfrak{a}_2 -\sum_{i=1}^{r} F_i
\end{equation}
The Schubert calculus tells us that $c_2(\mathscr{E}_2)$ is
represented by the Schubert cycle parameterizing all lines contained
in a hyperplane. This gives
\begin{equation}
 c_2(\mathscr{E})  =
\mathfrak{a}_1\times\mathfrak{a}_2.
\end{equation}
Here $\mathfrak{a}_1\in|\mathscr{L}_1|$ and
$\mathfrak{a}_2\in|\mathscr{L}_2|$ are general elements from the
corresponding complete linear systems. We are also viewing
$\mathfrak{a}_1\times C_2$, $C_1\times\mathfrak{a}_2$ and
$\mathfrak{a}_1\times\mathfrak{a}_2$ as cycles on $\Sigma$. Let
$$
\xymatrix{
 P \ar[r]^f\ar[d]_{\tilde\pi} &\PP(V)\\
 \Sigma &
}
$$
be the total family of lines over $\Sigma$. The morphism $\tilde\pi$
admits two distinguished sections $D_1\subset P$ and $D_2\subset P$
corresponding to points on $C_1$ and $C_2$ respectively. It is easy
to see that $D_1$ and $D_2$ meet above $F_i$. Let $D=D_1+D_2$. If
$s\in \Sym^3V$ is the homogeneous polynomial that defines $X$, then
$f^*s$ can be viewed as a section of
$\calO_P(3)=f^*\calO_{\PP(V)}(3)$, whose zero defines
$$
\divisor(f^*s)=D+S
$$
for some divisor $S\subset P$.

\begin{defn}\label{residue surface pair}
The surface $S$ together with the morphism $\phi=f|_S:S\rightarrow
X$ is called the \textit{residue surface} associated to the pair of
curves $C_1$ and $C_2$.
\end{defn}

In order to determine the divisor classes of $D$ and $S$, let
$\xi=c_1(\calO_P(1))$, where $\calO_P(1)=f^*\calO_{\PP(V)}(1)$. Let
$\pi=\tilde\pi|_D:D\rightarrow\Sigma$. Consider the following short
exact sequence
$$
\xymatrix@C=0.5cm{
  0 \ar[r] & \calO_P(-D) \ar[rr] && \calO_P \ar[rr] && \calO_D \ar[r] & 0 }
$$
By applying $R\tilde\pi_*$ to the above sequence, we get
$$
\xymatrix@C=0.5cm{
  0 \ar[r] & \calO_\Sigma=\tilde\pi_*\calO_P \ar[rr] && \pi_*\calO_D \ar[rr] && R^1\tilde\pi_*\calO_P(-D) \ar[r] & 0 }
$$
By duality,
$R^1\tilde\pi_*\calO_P(-D)=(\tilde\pi_*\calO(K_{P/\Sigma}+D))^*
=(\tilde\pi_*\calO_P(-2\xi+D+\tilde\pi^*c_1(\mathscr{E})))^*$. Hence
we get
$$
D=2\xi -\tilde\pi^*c_1(\mathscr{E})
-\tilde\pi^*c_1(\tilde\pi_*\calO_D)
$$
Note we have the following short exact sequence on $D$,
$$
\xymatrix@C=0.5cm{
  0 \ar[r] & \calO_{D_1}(-\sum F_i) \ar[rr] && \calO_D \ar[rr] && \calO_{D_2} \ar[r] & 0 }
$$
Here we view $F_i$ as divisors on $D_1$ by identifying $D_1$ with
$\Sigma$. By taking direct images, we have
$$
\xymatrix@C=0.5cm{
  0 \ar[r] & \calO_\Sigma(-\sum F_i) \ar[rr] &&\pi_*\calO_D \ar[rr] &&\calO_\Sigma \ar[r] & 0 }
$$
and it follows that $c_1(\pi_*\calO_D)=-\sum F_i$. This gives the
class of $D$ by
\begin{equation}\label{class of D 2}
D=2\xi-\tilde\pi^*(\mathfrak{a}_1\times C_2 +C_1\times
\mathfrak{a}_2) +2\sum_{i=1}^{r}\tilde\pi^*F_i
\end{equation}
Since $D+S=3\xi$, we get
\begin{equation}\label{class of S 2}
S=\xi+ \tilde\pi^*(\mathfrak{a}_1\times C_2 +C_1\times
\mathfrak{a}_2) -2\sum_{i=1}^{r}\tilde\pi^*F_i
\end{equation}

\begin{lem}
Counted with multiplicity, there are $5e_1e_2-6r$ secant lines of
the pair $(C_1,C_2)$, where $r$ is the number of points in which the
two curves meet each other.
\end{lem}
\begin{proof}
The surface $S$ gives rise to a section $s_0$ of $\calO_P(S)$. Then
$s_0$ can also be viewed as a section of $\tilde\pi_*\calO_P(S)$,
whose zero locus defines the scheme $B_{C_1,C_2}$ of secant lines of
$(C_1,C_2)$. By \eqref{class of S 2}, we easily see that
$$
\tilde\pi_*\calO(S)=\mathscr{E}\otimes p_1*\mathscr{L}_1\otimes
p_2^*\mathscr{L}_2\otimes\calO_\Sigma(-\sum F_i)
$$
Hence one easily computes
$$
c_2(\tilde\pi_*\calO_P(S))=5\mathfrak{a}_1\times\mathfrak{a}_2+6(\sum
F_i)^2\equiv 5e_1e_2-6r
$$
This implies that the length of $B_{C_1,C_2}$ is $5e_1e_2-6r$.
\end{proof}

\begin{defn}\label{well positioned pair}
For each secant line $l$ of $(C_1,C_2)$ we define the
\textit{multiplicity}, $e(C_1,C_2;l)$, of $l$ to be the length of
$B_{C_1,C_2}$ at the point $[l]$. The pair $(C_1,C_2)$ is called
\textit{well-positioned} if it has exactly $5e_1e_2-6r$ distinct
secant lines, namely all the secant lines are of multiplicity 1.
\end{defn}

\begin{prop}\label{prop of pair residue}
Let $C_1$ and $C_2$ be two smooth complete curves on $X$ as above.
Then the following are true.\\
(1) The residue surface $S$ associated to $(C_1,C_2)$ is smooth if
and only if the pair $(C_1,C_2)$ is well-positioned. In that case,
$S$ is a blow up of $\Sigma$ at $5e_1e_2-6r$ points.\\
(2) Assume that $(C_1,C_2)$ is a well-positioned pair. The following
equalities hold in $\mathrm{CH}_1(S)$
\begin{align}
 \xi|_S &=2\sigma^*(\mathfrak{a}_1\times
C_2 +C_1\times\mathfrak{a}_2)-3\sum_{i=1}^{r}\sigma^* F_i-\sum_{i=1}^{5e_1e_2-6r}E_i\\
 D|_S &=3\sigma^*(\mathfrak{a}_1\times
C_2 +C_1\times\mathfrak{a}_2)-4\sum_{i=1}^{r}\sigma^* F_i-2\sum_{i=1}^{5e_1e_2-6r}E_i\\
 S|_S &=3\sigma^*(\mathfrak{a}_1\times
C_2 +C_1\times\mathfrak{a}_2)-5\sum_{i=1}^{r}\sigma^*
F_i-\sum_{i=1}^{5e_1e_2-6r}E_i
\end{align}
where $\sigma=\tilde{\pi}|_S:S\rightarrow \Sigma$ is the blow up and
$E_i$ are the exceptional divisors of $\sigma$.
\end{prop}

\begin{proof}
The statement (1) follows from direct local computation as before.
For (2), note that under the assumption, $S$ is a smooth surface. By
adjunction formula,
$$
K_S=(K_P+S)|_S=\sigma^*K_\Sigma-\xi|_S+2\sigma^*(\mathfrak{a}_1\times
C_2 +C_1\times\mathfrak{a}_2)-3\sum_{i=1}^{r}\sigma^* F_i
$$
On the other hand the blow-up gives
$$
K_S=\sigma^*K_\Sigma +\sum_{i=1}^{5e_1e_2-6r}E_i
$$
We easily get $\xi|_S$ by comparing the above two equalities. The
other equalities follow easily.
\end{proof}

\section{Relations among 1-cycles} In this section we assume that
$X\subset\PP(V)=\PP^n_k$ is a smooth cubic hypersurface. We want to
investigate the cycle classes represented by a curve on $X$.

Let $\mathrm{Hilb}_{e,g}(X)$ be the Hilbert scheme of degree $e$
genus $g$ curves on $X$. Let
$\mathcal{H}^{e,g}=\mathcal{H}^{e,g}(X)$ be the normalization of the
reduced Hilbert scheme $(\mathrm{Hilb}_{e,g}(X))_{red}$. We use
$Z\subset X\times\mathcal{H}^{e,g}$ to denote the universal family
over $\mathcal{H}^{e,g}$. Note that $Z$ is nothing but the pull back
of the universal family over $\mathrm{Hilb}_{e,g}(X)$. Let
$\mathcal{U}^{e,g}\subset \mathcal{H}^{e,g}$ be the open subscheme
of all points $[C]$ whose corresponding curve $C$ is
well-positioned. When $\mathcal{U}^{e,g}\neq\emptyset$, we define
$$
Y_0=\Set{(x,[C])\in X\times\mathcal{U}^{e,g}}{x\in l\text{ for some
secant line }l\text{ of }C}
$$
and take $Y=\overline{Y_0}\subset X\times\mathcal{H}^{e,g}$. We need
to generalize the concept of secant lines.

\begin{defn}\label{generalized secant line}
Let $[C]\in\mathcal{H}^{e,g}$ be a closed point. If the fiber
$Y_{[C]}=\cup \tilde{E}_i$ is purely one dimensional. Let $E_i$ be
$\tilde{E}_i$ with the reduced structure. Then we call $E_i$ a
\textit{generalized secant line} of $C$ with \textit{multiplicity}
$e(C;E_i)$ being the length of $\tilde{E}_i$ at its generic point.
For a pair $(C_1,C_2)$ of curves on $X$, a \textit{generalized
secant line} $l$ of the pair is defined to be a generalized secant
line of $C_1\cup C_2$ with multiplicity $e(C_1,C_2;l)=e(C_1\cup
C_2;l)-e(C_1;l)-e(C_2;l)$. If the above multiplicity is 0, we do not
count $l$ as a generalized secant line of the pair.
\end{defn}

Let $\mathcal{T}^{e,g}\subset \mathcal{H}^{e,g}$ be the open subset
of all curves (including reducible ones) that have finitely many
generalized secant lines. Then we have $\mathcal{U}^{e,g}\subset
\mathcal{T}^{e,g}$. We use
$\mathcal{U}^{e,g}_{[C]}\subset\mathcal{T}^{e,g}_{[C]}
\subset\mathcal{H}^{e,g}_{[C]}$ to denote the corresponding subsets
of the component containing the point $[C]$.

\begin{thm}\label{main theorem}
Let $X\subset \PP^n_k$ be a smooth cubic hypersurface of dimension
at least 3. Let $h$ denote the class of hyperplane on $X$.\\
\indent(a) Let $[C]\in\mathcal{H}^{e,g}(X)$ be a connected curve of
degree $e$ and genus $g$ on $X$. Assume: (a1) $C$ is not a line;
(a2) $C$ has finitely many (generalized) secant lines $E_i$,
$i=1,\ldots,m$, with multiplicity $a_i=e(C;E_i)$. Then we have
\begin{equation}\label{relation 1}
(2e-3)C+\sum_{i=1}^{m} a_iE_i = (\frac{1}{2}(e-1)(3e-4)-2g)h^{n-2}
\end{equation}
 in $\mathrm{CH}_1(X)$.\\
\indent(b) Let $C_1$ and $C_2$ be two connected curves on $X$ with
degrees $e_1$ and $e_2$ respectively. Assume: (b1) $C_1$ and $C_2$
only meet each other transversally at $r$ smooth points
$x_1,\ldots,x_r$; (b2) any component of $C_1$ or $C_2$ containing
any of the $x_i$'s is not a line; (b3) $(C_1,C_2)$ has finitely many
(generalized) secant lines $E_i$, $i=1,\ldots,m$, with
multiplicities $a_i$. Then we have
\begin{equation}\label{relation 2}
2e_2C_1+2e_1C_2 +\sum_{i=1}^{m}a_iE_i=(3e_1e_2-2r)h^{n-2}
\end{equation}
in $\mathrm{CH}_1(X)$.\\
\indent(c) Let $L\subset X$ be a line and $C\subset X$ be an
irreducible curve of degree $e$ on $X$. Assume: (c1) $L$ meet $C$
transversally in a point $x$; (c2) $C$ has finitely many
(generalized) secant lines; (c3) the pair $(L,C)$ has finitely many
(generalized) secant lines $E_i$ with multiplicity $a_i$,
$i=1,\ldots,m$. Then
\begin{equation}\label{relation 3}
(2e-1)L+2C+\sum_{i=1}^{m}a_iE_i=(3e-2)h^{n-2}
\end{equation}
in $\mathrm{CH}_1(X)$.
\end{thm}

%\begin{rmk}
%The essential assumption in each statement is the finiteness on the
%number of (generalized) secant lines. All the other technical
%assumptions can be removed. The above generality is already enough
%for our applications.
%\end{rmk}

\begin{cor}\label{chow group}
For smooth cubic hypersurfaces of dimension at least 2, the Chow
group $\mathrm{CH}_1(X)$ is generated by the classes of all lines on
$X$. If $\dim X\geq3$, then algebraic equivalence and homological
equivalence are the same for 1-cycles on $X$. For $\dim X\geq 5$, we
have $\mathrm{CH}_1(X)\cong\Z$.
\end{cor}

\begin{proof}
(of Theorem \ref{main theorem}) First, we will prove both statements
(a) and (b) in a similar way, under the additional assumption that
the curve $C$ (resp. $C_1$ and $C_2$) is (are) smooth and $C$ (resp.
the pair $(C_1,C_2)$) is well-positioned. Consider the following
situation. Let $\Sigma$ be a surface together with a morphism
$\varphi:\Sigma\rightarrow \mathrm{G}(2,V)$ and a diagram as follows
$$
\xymatrix{
  D\cup S \ar[d]_{i'} \ar[r]^{\;\varphi'} & Y \ar[d]_{i_Y} \ar[r]^{p'} & X \ar[d]^{i_X} \\
  P \ar[d]_{\tilde\pi} \ar[r]^{\tilde\varphi\quad} & \mathrm{G}(1,2,V) \ar[d]_{\pi_0} \ar[r]^{\quad p} & \PP(V) \\
  \Sigma \ar[r]^{\varphi\quad} & \mathrm{G}(2,V)  &   }
$$
where all the squares are fiber products. Let $\phi:S\rightarrow X$
and $\phi':D\rightarrow X$ be the natural morphisms. Assume that the
image of $D$ in $X$ is at most 1 dimensional and the image of $S$ in
$X$ is 2 dimensional. Then one computes
$$
\phi_*S=p'_*(D+
S)=p'_*i_Y^*\tilde\varphi_*P=i_X^*p_*\tilde\varphi_*P=(p_*\tilde\varphi_*P)|_X
$$
Note that $p_*\tilde\varphi_*P=a[\PP^3]$ in $\mathrm{CH}_3(\PP(V))$
for some integer $a>0$. It follows that $\phi_*S=a\,h^{n-3}$. Let
$\xi=c_1(\calO_P(1))$ where
$\calO_P(1)=\tilde\varphi^*p^*\calO_{\PP(V)}(1)$. Then by projection
formula, we get
\begin{equation}\label{image of xi}
\phi_*(\xi|_S)=\phi_*\phi^*h=h\cdot\phi_*S=a\,h^{n-2}
\end{equation}
Namely $\phi_*\xi|_S$ is an integral multiple of $h^{n-2}$. Let
$\mathfrak{A}$ be a divisor on $\Sigma$. By construction
$$
\varphi'_*(\tilde\pi^*\mathfrak{A}|_D+\tilde\pi^*\mathfrak{A}|_S)=\tilde\varphi_*\tilde\pi^*\mathfrak{A}|_Y=\pi^*\varphi_*\mathfrak{A}|_Y
$$
By applying $p'_*$, we get the following
\begin{equation}\label{class from Sigma}
\phi'_*(\tilde\pi^*\mathfrak{A}|_D)+\phi_*(\tilde\pi^*\mathfrak{A}|_S)
=p'_*i_Y^*\pi^*\varphi_*\mathfrak{A}=(p_*\pi_0^*\varphi_*\mathfrak{A})|_X
= b[\PP^2]|_X=b\,h^{n-2}
\end{equation}
for some integer $b$.

To prove (a), we take $\Sigma$ to be $C^{(2)}$ and $S$ to be the
associated residue surface. We know that $D\cong C\times C$ and
$\phi'=p_1$. We take $\mathfrak{A}=D_\mathfrak{a}$ in \eqref{class
from Sigma}, where $\mathfrak{a}\subset C$ a general hyperplane
section of $C$.  Recall that
$D_\mathfrak{a}=(\tilde{\pi}|_D)_*(\fa\times C)$ is defined in
\eqref{first notations}. Note that $\phi'_*(\tilde\pi^*
D_\mathfrak{a}|_D)=\phi'_*(\fa\times C+C\times\fa)=e[C]$, we know
that
$$
\phi_*(\tilde\pi^* D_\mathfrak{a}|_S)=-e[C]+(\cdots)h^{n-2}
$$
Similarly, in \eqref{class from Sigma} if we take
$\mathfrak{A}=\Delta_0$ (see section 2 for the definition of
$\Delta_0$) and note that $\phi'_*(\tilde\pi^*\Delta_0|_D)=-[C]$, we
get
$$
\phi_*(\tilde\pi^*\Delta_0|_S)=[C]+(\cdots)h^{n-2}
$$
Put all these in to the equality for $\xi|_S$ in Proposition
\ref{prop of single residue}, we have
$$
(2e-3)[C]+\sum E_i=-\phi_*(\xi|_S)+(\cdots)h^{n-2}=\text{multiple of
}h^{n-2}
$$
We know that the degree of $h^{n-2}$ is 3, hence by comparing the
total degrees of the two sides of the above equality, we see that
the right hand side is $(1/2(e-1)(3e-4)-2g)h^{n-2}$.

For (b), we take $\Sigma$ to be the blow-up of $C_1\times C_2$ at
the points where the two curves meet. We use the notations from
Proposition \ref{prop of pair residue}. First we notice
$\sigma^*F_i$ maps to a singular plane cubic whose class is
$h^{n-2}$. This is because the image of $\sigma^*F_i$ in $X$ is the
intersection of $\Pi_i$ and $X$. Then we take
$\mathfrak{A}=\mathfrak{a}_1\times C_2 + C_1\times\mathfrak{a}_2$ in
\eqref{class from Sigma} and note that
$\phi'_*(\tilde\pi^*\mathfrak{A}|_D)=e_2[C_1]+e_1[C_2]$, we get
$$
\phi_*(\sigma^*\mathfrak{A})=-e_2[C_1]-e_1[C_2]+(\cdots)h^{n-2}
$$
We put this relation into the equality for $\xi|_S$ in Proposition
\ref{prop of pair residue} and get
$$
2e_2[C_1]+2e_1[C_2]+\sum E_i=(\cdots)h^{2}
$$
By comparing the degrees, we know that the right hand side is
$(3e_1e_2-2r)h^{n-2}$.

Note that if the statement of (b) holds for each pair
$(C_{1,i},C_{2,j})$ where $C_{1,i}$ is a component of $C_1$ and
$C_{2,j}$ is a component of $C_2$, then the (b) holds for
$(C_1,C_2)$. This means that we proved (b) when all components if
$C_1$ and $C_2$ are smooth and the components are pairwise
well-positioned.

To prove the theorem in full generality, we need the following
\begin{lem}\label{continuity lemma}
Let $B/k$ be a smooth projective curve. Let $y\in Z_{p+1}(X\times
B)$ be an algebraic cycle of relative dimension $p$ (over $B$). Let
$\gamma:B(k)\rightarrow \mathrm{CH}_p(X)$ be the map induce by $y$.
Assume that there exists a Zariski open $U\subset B$ such that
$\gamma$ is constant on $U$, then $\gamma$ is constant on $B$.
\end{lem}
The proof of this lemma is easy since any closed point of $B$ is
rationally equivalent to a linear combination of points of $U$.

\textit{Observation I}: If $[C]\in\mathcal{T}^{e,g}$ is a smooth
point and $\mathcal{U}^{e,g}_{[C]}\neq\emptyset$, then (a) holds for
$C$.

Set $z=(2e-3)Z+Y\in \mathrm{CH}^{n-2}(X\times\mathcal{H}^{e,g})$. We
can pick a smooth curve $B$ that maps to $\mathcal{H}^{e,g}_{[C]}$
and connects $[C]$ with a general point. Then we apply Lemma
\ref{continuity lemma} with $y$ being the pullback of $z$. Since (a)
holds for a general point of $B$, then the lemma implies that it
holds for all points of $B$, whose image are smooth points of
$\mathcal{T}^{e,g}_{[C]}$. Completely similarly, we have the
following

\textit{Observation II}: Notations as in (b). Let
$[C_1]\in\mathcal{F}_{C_1,C_2}$ be a smooth point, where
$\mathcal{F}_{C_1,C_2}\subset\mathcal{H}^{e,g}_{[C_1]}$ is the
subscheme of curves meeting $C_2$ in $r$ smooth points. If (b) holds
for $(C'_1,C_2)$ for a general point
$[C'_1]\in\mathcal{F}_{C_1,C_2}$, then (b) holds for $(C_1,C_2)$.

We first prove (b) in it's full generality. Step 1: we assume that
$C_2$ has smooth components. Without loss of generality, we assume
that $C_2$ is smooth irreducible. We can attach sufficiently many
very free rational curves $C'_i$ to $C_1$ at smooth points and get
$C'=C_1\bigcup(\cup_iC'_i)$ of degree $e'$ and genus $g'=g_1$, such
that (i) The deformation of $C'$ in $X$, fixing the points where the
curve meets $C_2$, is unobstructed, see \cite{Ko} II.7.; (ii) The
pairs $(C'_i,C_2)$ are all well positioned; (iii) For a general
point $[\tilde{C}]\in \mathcal{F}_{C',C_2}$, the pair
$(\tilde{C},C_2)$ is well-positioned. Note that (i) implies that
$[C']\in\mathcal{F}_{C',C_2}$ is a smooth point. Furthermore, (ii)
and (iii) imply that (b) holds for $(C'_i,C_2)$ and
$(\tilde{C},C_2)$. By observation II, we know that (b) holds for
$(C',C_2)$. Then by linearity, we know that (b) holds for
$(C_1,C_2)$. Step 2: general case. We can repeat the above argument
by attaching very free rational curve to one of the curves.

So for (a) with $C$ being general, the only case we need to verify
is when $\mathcal{U}^{e,g}=\emptyset$ or $[C]\in\mathcal{H}^{e,g}$
is not a smooth point. In either case we can attach very free
rational curves $C_i$ to $C$ at smooth points and get a curve
$C'=C\bigcup(\cup_iC_i)$. We may assume that the following
conditions holds: (i) The deformation of the curve $ C'$ in $X$ is
unobstructed and hence gives a smooth point
$[C']\in\mathcal{H}^{e',g'}$; (ii) the pair $(C,C_i)$ has finitely
many generalized secant lines $E_{i,s}$ of multiplicity $a_{i,s}$;
(iii) Each pair $(C_i,C_j)$ is well-positioned with secant lines
$L_{ij,s}$; (iv) Each curve $C_i$ is well-positioned with secant
lines $L_{i,s}$; (v) A general point
$[\tilde{C}]\in\mathcal{H}^{e',g'}_{[C']}$ corresponds to a
well-positioned curve $\tilde{C}\subset X$. By observation I, (i)
and (v) imply that (a) holds for $C'$, i.e.
$$
(2e+2\sum_i e_i-3)(C+\sum_i C_i)+\sum_{i}\sum_s L_{i,s}
+\sum_{i,j}\sum_s L_{ij,s} +\sum_i\sum_s a_{i,s}E_{i,s}
+\sum_{i=1}^{m} a_iE_i
$$
is a multiple of $h^{n-2}$. By (iv), we have
$$
(2e_i-3)C_i+\sum_s L_{i,s} = (\cdots)h^{n-2}
$$
By (ii), we have
$$
2e_iC+2eC_i+\sum_s a_{i,s}E_{i,s} = (\cdots)h^{n-2}
$$
By (iii), we have
$$
2e_iC_j+2e_jC_i+\sum_s L_{ij,s} = (\cdots)h^{n-2}, \quad i<j
$$
Let $i$ (and $j$) run over all possible choices and sum up the above
three equations, then take the difference of that with the first big
sum above, we get (a) for the curve $C$.

To prove (c), let $L_i$ be the (generalized) secant lines of $C$ and
$b_i$ be the corresponding multiplicities. Then by (a) for $L\cup
C$, we have
$$
(2e-1)(L+C)+\sum b_iL_i +\sum a_iE_i=(\cdots)h^{n-2}
$$
By (a) for the curve $C$, we get
$$
(2e-3)C+\sum b_iL_i = (\cdots)h^{n-2}
$$
Then we take the difference of the above two equations and get (c).
\end{proof}

\begin{proof}
(of Corollary \ref{chow group}). Let $C\subset X$ be an irreducible
curve on $X$. First assume that $C$ is smooth and has finitely many
secant lines. Then (a) implies that $(2e-3)C$ is an integral
combination of lines. We choose a line $L$ such that $(C,L)$ has
finitely many secant lines, then (b) implies that $2C$ is an
integral combination of lines. Hence $C$ itself is an integral
combination of lines in $\mathrm{CH}_1(X)$. If $C$ has infinitely
many secant lines. Then we attach sufficiently many very free
rational curve $C_i$ to $C$ and get $C'=C\bigcup(\cup_iC_i)$ such
that (i) the class each $C_i$ can be written as integral combination
of lines; (ii) The deformation of $C'$ in $X$ is unobstructed; $C'$
can be smoothed out to get a complete family $\pi:S\rightarrow B$ to
a smooth complete curve $B$; (iii) $S_{b_0}$ is the curve $C'$ and a
general fiber $S_b$ is well-positioned on $X$. By associating the
secant lines of $S_b$ to the point $b$, we get a rational map
$\varphi_0:B\dasharrow \Sym^{n_1}(F)$, where $F=F(X)$ is the Fano
scheme of lines on $X$. Pick a general line $L$ on $X$. By
associating the secant lines of $(L,S_b)$ to $b$, we have a rational
map $\phi_0:B\dasharrow \Sym^{n_2}(F)$. Both of the two maps can be
extended to the whole curve $B$ and get morphisms $\varphi$ and
$\phi$. Assume that $\varphi(b_0)=\sum a_i[L_i] $ and
$\phi(b_0)=\sum a'_i[E_i]$. Then the Lemma \ref{continuity lemma}
implies the following equalities hold in the Chow group,
$$
(2e'-3)C'+\sum a_iL_i = (\cdots)h^{n-2}
$$
and
$$
2C'+(2e'-1)L +\sum b_iE_i = (\cdots)h^{n-2}
$$
Hence $C'$ is an integral combination of lines. This implies that
$C$ itself is so.

When $n\geq 4$, the scheme, $F=F(X)$, of lines on $X$ is smooth
irreducible. This shows that algebraic equivalence agrees with
homological equivalence for 1-cycles on $X$. When $n\geq 6$, the
scheme $F=F(X)$ is a smooth Fano variety, see Theorem \ref{geomtry
of fano scheme}. This implies that $F$ is rationally connected, see
\cite{kmm} and \cite{cam}. Hence all lines on $X$ are rationally
equivalent to each other.
\end{proof}

\section{Fano scheme of lines and intermediate Jacobian of a cubic threefold}
In this section we first review some geometry of the Fano scheme of
lines on a cubic hypersurface and then collected some known results
about intermediate Jacobian of a cubic threefold.
\subsection{General theory on Fano schemes}
Let $X\subset\PP^n_k$ be a smooth cubic hypersurface over an
algebraically closed field. Let $F=F(X)$ be the Fano scheme of lines
on $X$. Let $V=\mathrm{H}^0(\PP^n,\calO(1))$. By definition, $F(X)$
can be viewed as a subscheme of $\mathrm{G}(2,V)$, the grassmannian
that parameterizes all lines on $\PP^n$. Let
\begin{equation}\label{universal line}
\xymatrix{
    \mathrm{G}(1,2,{V})\ar[r]^{f}\ar[d]_{\pi} &\mathrm{G}(1,{V})=\PP^n\\
    \mathrm{G}(2,V) &
  }
\end{equation}
be the universal family of lines on $\PP^n=\PP(V)$. Let
$s\in\Sym^3V$ be the homogeneous polynomial defining $X$. Then $s$
can also be viewed as a section $\tilde{s}$ of
$\pi_*f^*\calO_{\PP^n}(3)=\Sym^3(\mathscr{E}_2)$, where
$\mathscr{E}_2$ is the universal rank two quotient bundle on
$\mathrm{G}(2,V)$. Then $F\subset\mathrm{G}(2,V)$ is exactly the
zero locus of $\tilde{s}$.
\begin{thm}\label{geomtry of fano scheme}
(\cite{ak}) Notations as above, then the following are true\\
(a) $F$ is smooth irreducible of dimension $2n-6$;\\
(b) The dualizing sheaf $\omega_F\cong\calO_F(5-n)$, here
$\calO_F(1)$ comes from the Pl\"ucker embedding of
$\mathrm{G}(2,V)$. In particular, if $n\geq 6$ then $F$ is a Fano
variety.
\end{thm}

\subsection{Intermediate Jacobian of a cubic threefold}
In this subsection we fix $X\subset \PP^4_{k}=\PP(V)$ being a smooth
cubic threefold and we assume that $\ch(k)\neq2$. Let $F=F(X)$ the
Fano surface of lines on $X$. Let
\begin{equation}\label{universal line cubic threefold}
\xymatrix{P\ar[r]^{f}\ar[d]_\pi &X\\
  F &}
\end{equation}
be the universal family of lines on $X$. Let $l\subset X$ be a line
on $X$ and $[l]$ be the corresponding point on $F$. Let
$\mathscr{N}_{l/X}$ be the normal bundle. Then we have the following
short exact sequence
$$
\xymatrix@C=0.5cm{
  0 \ar[r] & \mathscr{N}_{l/X} \ar[rr] && \mathscr{N}_{l/\PP^4} \ar[rr] && \mathscr{N}_{X/\PP^4}|_l \ar[r] & 0
  }.
$$
Since $\mathscr{N}_{l/\PP^4}\cong \calO(1)^{\oplus 3}$ and
$\mathscr{N}_{X/\PP^4}|_l\cong \calO(3)$, we get
$\mathscr{N}_{l/X}\cong \calO^2$ or $\calO(1)\oplus \calO(-1)$. The
line $l$ is of \textit{first type} if
$\mathscr{N}_{l/X}\cong\calO^2$; otherwise, $l$ is of \textit{second
type}. Let $D_0\subset F$ be the locus whose points are all lines of
second type. The condition for $l$ to be of second type is
equivalent to the existence of a plane $\Pi\subset\PP^4$ such that
$\Pi\cdot X=2l+l'$ for some other line $l'\subset X$, see Lemma 1.14
of \cite{murre}. A point $x\in X$ is called an \textit{Eckardt
point} if there are infinitely many lines through $x$. We know that
there are at most finitely many of such points, see the discussion
on p.315 of \cite{cg}.

\begin{prop}(\cite{cg}, \cite{murre}, \cite{ak})\label{geometry of Fano
surface}
The following are true\\
(a) We have a canonical isomorphism
$\Omega_F^1\cong\mathscr{E}_2|_F$;\\
(b) $D_0$ is a smooth curve on $F$ whose divisor class is $2K_F$;\\
(c) The ramification divisor of $f$ in diagram \eqref{universal line
cubic threefold} is given by $R=\pi^{-1}D_0\subset P$. If we write
$B=f(R)\subset X$, then $B$ is linearly equivalent to $30h$, where
$h$ is a hyperplane section of $X$;\\
(d) We have the following equalities
\begin{align*}
&h^0(F,\calO_F)=1, \qquad h^1(F,\calO_F)=5, \qquad h^2(F,\calO_F)=10,\\
&h^0(F,\omega_F)=10, \qquad h^1(F, \omega_F)=5, \qquad
h^2(F,\omega_F)=1
\end{align*}
\end{prop}

\begin{proof}
(a) is proved in \cite{ak} (Theorem 1.10); (d) is proved in
\cite{ak} (Proposition 1.15). For (b), we note that smoothness of
$D_0$ is Corollary 1.9 of \cite{murre} and the class of $D_0$ is
computed in \cite{cg} (Proposition 10.21). If $[l]\in D_0$, then
$\mathrm{H}^0(l,\mathscr{N}_{l/X})$ does not generate
$\mathscr{N}_{l/X}$. This implies that infinitesimally $l$ only
moves in the $\calO(1)$ direction of $\mathscr{N}_{l/X}$. Hence $f$
ramifies along $\pi^{-1}([l])$. If $[l]\notin D_0$, then
$\mathscr{N}_{l/X}$ is globally generated and $f$ is \'etale along
$\pi^{-1}([l])$. Since $X$ has Picard number one, to know the
divisor class of $B$, we only need to compute the intersection
number $B\cdot l$ for a general line $l\subset X$. By the projection
formula,
$$
B\cdot l=(f_*\pi^*D_0)\cdot l=D_0\cdot(\pi_*f^*l)=2K_F\cdot D_l,
$$
where $D_l\subset F$ is the divisor of all lines meeting $l$. It is
known (see \cite{cg}, Section 10) that $K_F$ is numerically
equivalent to $3D_l$ and $(D_l)^2=5$. Hence we have $B\cdot l=30$.
This proves (c).
\end{proof}

\begin{defn}
Let $T/k$ be a variety, then the group of \textit{divisorial
correspondence}, $\mathrm{Corr}(T)$ is defined to be
$$
\mathrm{Corr}(T)=\frac{\Div(T\times T)}{p_1^*\Div(T)+p_2^*\Div(T)}
$$
\end{defn}

For the Fano surface $F$, there is a natural divisorial
correspondence $I\in \mathrm{Corr}(F)$, called \textit{incidence
correspondence}. To be more precise, we have the universal line
$P\subset X\times F$. Let $p$ and $q$ be the two projections from
$X\times F\times F$ to $X\times F$ and $r$ be projection to $F\times
F$. Then
$$
I=r_*(p^*P\cdot q^*P-P_0)
$$
where $P_0=\Set{(x,[l],[l])}{x\in l}$. Equivalently, $I$ is the
closure of all pairs $(l_1,l_2)$ of distinct lines such that $l_1$
meets $l_2$. Then $I$ induces a homomorphism
$\lambda_I:\Alb(F)\rightarrow\Pic^0_F$.

Let $\mathrm{A}_1(X)$ be the group of algebraic 1-cycles on $X$ that
are algebraically equivalent to 0 modulo rational equivalence. Let
$\varphi: \mathrm{A}_1(X)\rightarrow A$ be a homomorphism from
$\mathrm{A}_1(X)$ to an abelian variety $A$. The homomorphism
$\varphi$ is said to be \textit{regular} if for any algebraic family
$Z\subset X\times T$ of curves on $X$ parameterized by $T$, the rule
$t\mapsto\varphi([Z_t]-[Z_{t_0}])$ induces a morphism from $T$ to
$A$. The homomorphism $\varphi$ is said to be \textit{universal} if
for any regular homomorphism $\varphi':\mathrm{A}_1(X)\rightarrow
A'$ to an abelian variety $A'$, there is a unique homomorphism
$\psi:A\rightarrow A'$ of abelian varieties such that
$\varphi'=\psi\circ\varphi$. By the main result of \cite{murre3},
there is a universal regular homomorphism
$\varphi_0:\mathrm{A}_1(X)\rightarrow J(X)$ from $\mathrm{A}_1(X)$
to an abelian variety $J(X)$. In \cite{murre2}, $J(X)$ is naturally
realized as a Prym variety. To be more precise, let $l\subset X$ be
a general line on $X$ and $\tilde\Delta$ be the smooth curve of
genus 11 that parameterizes all lines on $X$ meeting $l$. We use
$S_l$ to denote the surface swept out by lines parameterized by
$\tilde{\Delta}$. There is natural involution $\sigma$ on
$\tilde\Delta$. Here $\sigma([l_1])$ is defined to be the residue
line $[l_2]$ of $l\cup l_1$. Let $\Delta=\tilde\Delta/\sigma$ be the
quotient curve which is a smooth plane quintic, hence of genus 6.
The involution $\sigma$ induces an involution
$i:J(\tilde\Delta)\rightarrow J(\tilde\Delta)$. The Prym variety
associated to $\tilde\Delta/\Delta$ is defined to be
$\mathrm{Pr}(\tilde\Delta/\Delta)=\mathrm{Im}(1-i)\subset
J(\tilde\Delta)$. Theorem 5 of \cite{murre2} says that
$J(X)\cong\mathrm{Pr}(\tilde{\Delta}/\Delta)$. There is a natural
principal polarization on $\mathrm{Pr}(\tilde\Delta/\Delta)$ whose
theta divisor $\Xi$ is half of the restriction of the theta divisor
from $J(\tilde\Delta)$, see part iii of \cite{mumford}. This induces
a principal polarization on $J(X)$ whose theta divisor will be
denoted by $\Theta$. It is shown in Section IV of \cite{murre2} that
the principal polarization $\Theta$ essentially comes from the
Poincar\'e duality on $\mathrm{H}^3(X)$. Here
$\mathrm{H}^3(X)=\mathrm{H}^3(X,\Q_\ell)$ is the $\ell$-adic
cohomology with $\ell\neq\mathrm{char}(k)$.

\begin{defn}\label{defn of intermediate jacobian}
We define the \textit{intermediate Jacobian} of $X$ to be $J(X)$
together with the principal polarization $\Theta$.
\end{defn}

\begin{rmk}
When $k=\C$ the above definition agrees with the classical
definition using Hodge theory, see \cite{cg}. See \cite{murre3} for
a proof of this fact.
\end{rmk}

Suppose that we are given a family $Z\rightarrow T$ of algebraic
1-cycles on $X$. We have a natural homomorphism
$\Psi_T:\mathrm{A}_0(T)\rightarrow \mathrm{A}_1(X)$ by sending a
point of $T$ to the class of the 1-cycle it represents. Since
$\varphi_0:\mathrm{A}_1(X)\rightarrow J(X)$ is regular, the
composition $\varphi_0\circ\Psi_T:T\rightarrow J(X)$ is a morphism
and hence it induces a homomorphism
$$
\psi_T:\mathrm{Alb}(T)\rightarrow J(X)
$$
of abelian varieties. Both $\Psi_T$ and $\psi_T$ will be called the
\textit{Abel-Jacobi map}. In particular, we have the Abel-Jacobi
maps $\psi_F$ and $\Psi_F$ associated to the Fano surface $F$.

\begin{thm}(\cite{cg},\cite{murre2},\cite{beauville1},\cite{beauville2})
Let $X/k$ be a smooth cubic threefold and $\ch(k)\neq2$. Let
$F=F(X)$ be the Fano surface of lines on $X$. Then the following are true.\\
(a) The incidence correspondence $I\in\mathrm{Corr}(F)$ defines a
principal polarization on $\Alb(F)$. Namely,
$\lambda_I:\Alb(F)\rightarrow\Pic^0_F$ is an isomorphism.\\
(b) The homomorphism $\varphi_0:\mathrm{A}_1(X)\rightarrow J(X)$ is
an isomorphism of abelian groups. \\
(c) The homomorphism $\Psi_F$
induces a isomorphism
$$
\psi_F:\Alb(F)\rightarrow J(X)
$$
as principally polarized abelian varieties.\\
(d) $X$ is completely determined by $J(X)$ together with the
polarization.
\end{thm}
\begin{proof}
For (a), see Theorem 8 in \cite{murre2}. We use the realization of
$J(X)$ as $\mathrm{Pr}(\tilde\Delta/\Delta)$. Let $X_l$ be the
blow-up of $X$ along the general line $l\subset X$. Then the
projection from $l$ defines a morphism $p:X_l\rightarrow\PP^2$. This
makes $X_l$ into a conic bundle over $\PP^2$ and
$\tilde\Delta/\Delta$ is the associated double cover. Then by
Th\'eor\`eme 3.1. of \cite{beauville1}, we get
$\mathrm{Pr(\tilde\Delta/\Delta)}\cong\mathrm{A}_1(X_l)$. Since
$l\cong\PP^1$, we also have $\mathrm{A}_1(X_l)\cong\mathrm{A}_1(X)$.
Hence (b) follows easily. For (c), it is known that $\psi_F$ is an
isomorphism, see Theorem 7 of \cite{murre2}. Hence we only need to
track the polarizations. Fix $s_0\in F$ then we have the Albanese
morphism $\alpha:F\rightarrow \Alb(F)$ which sends $s_0$ to 0. Let
$f=\psi_F\circ\alpha$. Then by Lemma 7 of \cite{murre2}, we know
that $(f\times f)^*\tilde\Theta\equiv I$, where $\tilde\Theta\subset
J(X)\times J(X)$ is the divisorial correspondence on $J(X)$ that
induces the principal polarization. By passing to the Albanese
variety, the above fact says that the polarization on $J(X)$ pulls
back to the polarization $\lambda_I$. The statement (d), also known
as ``Torelli theorem", was first obtained in \cite{cg} for $k=\C$
and then generalized to the general case in \cite{beauville2}.
\end{proof}

\section{Intermediate jacobian of a cubic threefold as Prym-Tjurin
variety} In this section, we use the relations among 1-cycles to
study the intermediate jacobian of a smooth cubic threefold $X$. To
be more precise, we show that the intermediate jacobian is naturally
isomorphic to the Prym-Tjurin variety constructed from curves on
$X$. Throughout this section we fix $X\subset\PP^4_k$ to be a smooth
cubic threefold over an algebraically closed field. Let $F=F(X)$ be
the Fano surface of lines on $X$ and $I\subset F\times F$ be the
incidence correspondence.
\begin{defn}
Let $\Gamma$ be a smooth curve which might be reducible. Let
$J=J(\Gamma)$ be jacobian of $\Gamma$. Let $i:J\rightarrow J$ be an
endomorphism which is induced by a correspondence. Assume that $i$
satisfies the following quadratic equation
\begin{equation}\label{quadratic relation}
(i-1)(i+q-1)=0
\end{equation}
for some integer $q\geq 1$. Then we define the \textit{Prym-Tjurin
variety} associated to $i$ as follows
\begin{equation}\label{definition of prym-tjurin}
\mathrm{Pr}(C,i)=\mathrm{Im}(1-i)\subset J(C)
\end{equation}
\end{defn}

Now Let $C$ be a possibly reducible smooth curve on the cubic
threefold $X$. Assume that $C$ has only finitely many secant lines.
Let $\tilde{C}$ be the normalization of the curve that parameterizes
all lines meeting $C$. Hence we have a natural morphism
$\eta:\tilde{C}\rightarrow F$. Let $D(I)=(\eta\times\eta)^*I$ be the
pull back of the incidence correspondence of $F$. Let
$i_0:\tilde{J}=J(\tilde{C})\rightarrow\tilde{J}$ be the endomorphism
induced by $D(I)$. Let $[L]\in \tilde{C}$ be a general point and
$L_i$ be all the secant lines of the pair $(L,C)$. Let $U\subset
\tilde{C}$ be the dense open subset of lines $L$ such that $(L,C)$
has finitely many secant lines $L_i$. The association
$$
[L]\mapsto \sum [L_i]
$$
defines a correspondence $D_U$ on the dense open subset $U$ of
$\tilde{C}$. Let $D\subset \tilde{C}\times\tilde{C}$ be the closure
of $D_U$.

\begin{defn}
A smooth complete curve $C\subset X$ is \textit{admissible} if the
following conditions hold: (i) $C$ has only finitely many secant
lines; (ii) if a line $L$ meets $C$, it meets $C$ transversally;
(iii) $C$ meets the divisor $B\subset X$ (see Proposition
\ref{geometry of Fano surface}) transversally in smooth points of
$B$.
\end{defn}
The following lemma will be needed in later proofs
\begin{lem}\label{six lines}
Let $x\in X$ be a general point and $L_j$, $j=1,\ldots,6$, be the
six lines on $X$ passing through $x$. Then
$$
\sum_{j}^{6}L_j=2h^2
$$
in $\mathrm{CH}_1(X)$, where $h$ is the hyperplane class on $X$.
\end{lem}
\begin{proof}
Fix a line $L=L_6$ passing through $x$. Let $E_t$, $t\in T$, be a
1-dimensional family of lines on $X$ such that $E_{t_0}=L_5$. By
choosing $E_t$ general enough, we may assume that $(L,E_t)$ is
well-positioned for $t$ close to but different from $t_0$. Let
$L_{t,j}$, $j=0,\ldots,4$, be the secant lines of $(L,E_t)$, $t\neq
t_0$. By (b) of Theorem \ref{main theorem}, we have
$$
2E_t+2L+\sum_{j=0}^{4}=3h^2
$$
Not let $t\to t_0$ and assume that $L_{t,0}$ specializes to the
residue line $L_0$ of $L_5\cup L_6$ and $L_{t,j}$ specializes to
$L_j$ for $j=1,2,3,4$. By taking limit, see \S11.1 of \cite{ful},
the above identity gives
$$
2L_5+2L_6+L_0+\sum_{j=1}^{4}=3h^2
$$
Since $L_0+L_5+L_6=h^2$, the lemma follows.
\end{proof}

\begin{thm}\label{J(X) as prym-tjurin}
Assume that $\ch(k)\neq 2$. Let $C=\cup C_s\subset X$ be an
admissible smooth curve of degree $e$. Assume that all components of
$C$ are rational. Let $\tilde{C}$ be the normalization of the curve
parameterizing lines meeting $C$ as above. Let
$i=i_0+1\in\End(\tilde{J})$. Then the following are
true.\\
(a) The endomorphism $i$ satisfies the quadratic relation
\eqref{quadratic relation} with $q=2e$. \\
(b) The correspondence $D\in\mathrm{Corr}(\tilde{C})$ is symmetric,
effective and without fixed point. The endomorphism $i$ is induced
by $D$.\\
(c) Assume that $\ch(k)\nmid q$. Then the associated Prym-Tjurin
variety $\mathrm{Pr}(\tilde{C},i)$ carries a natural principal
polarization whose theta divisor $\Xi$ satisfies the following
equation
$$
\tilde\Theta|_{\mathrm{Pr}(\tilde{C},i)}\equiv q\Xi
$$
where $\tilde\Theta$ is the theta divisor of $\tilde{J}$. \\
(d) Assume that $\ch(k)\nmid q$. Then the Abel-Jacobi map
$\psi_{\tilde{J}}:\tilde{J}\rightarrow J(X)$ factors through
$\mathrm{Pr}(\tilde{C},i)$ and induces an isomorphism
$$
u_C:\mathrm{Pr}(\tilde{C},i)\rightarrow J(X)
$$
as principally polarized abelian varieties.
\end{thm}

\begin{proof}
To prove (a), we consider the following commutative diagram
\begin{equation}\label{diagram prym-tjurin}
\xymatrix{
 \tilde{J}\ar[r]^{i_0+q} & \tilde{J}\ar[r]^{\eta_*\quad}\ar[dd]^{i_0} &\Alb(F)\ar[rd]^{\psi_F}\ar[dd]^{\lambda_I}
 &\\
 & & &J(X)\ar[ld]\\
 &\tilde{J} &\Pic_F^0\ar[l]_{\eta^*} &
}
\end{equation}
Here the map $J(X)\rightarrow \Pic_F^0$ is induced by the map
$\mathrm{A}_1(X)\rightarrow \Pic(F)$ which sends a curve on $X$ to
the divisor (on $F$) of all lines meeting the curve. Let
$[L],[L']\in \tilde{C}$ be two general points from the same
component such that the pair $(C,L)$ (resp. $(C,L')$) is
well-positioned. Let $E_j$ (resp. $E'_j$), $j=1,\ldots,5e-5$, be all
the secant lines of the pair $(C,L)$ (resp. $(C,L')$). Let $L_j$
(resp. $L'_j$), $j=1,\ldots,5$, be the other five lines on $X$
passing through $C\cap L$ (resp. $C\cap L'$). Then by
\eqref{relation 3} (and also \eqref{relation 2} if $C$ is
disconnected), we have
$$
(2e-1)L+\sum_{j=1}^{5e-5}E_j=(2e-1)L'+\sum_{j=1}^{5e-5}E'_j.
$$
Note that $i_0([L])=D(I)([L])=\sum_{j=1}^{5} [L_j]+\sum [E_j]$ and
that $L+\sum_{j=1}^{5} L_j=2h^2$, see Lemma \ref{six lines}. Then we
get
\begin{align*}
\psi_F\circ\eta_*\circ(i_0+q)([L]-[L']) &= 2e(L-L')+ (\sum L_j-\sum
L'_j) + (\sum E_j-\sum E'_j)\\
&= 2e(L-L') + (L'-L) + (1-2e)(L-L')\\
&=0
\end{align*}
Since $\psi_F$ is isomorphism, this implies that
\begin{equation}\label{condition b of kanev}
\eta_*\circ (i_0+q)=0
\end{equation}
Hence
$$
i_0(i_0+q)=\eta^*\circ\lambda_I\circ\eta_*\circ(i_0+q)=0,
$$
which is the same as the quadratic equation \eqref{quadratic
relation}. By construction, we have
$$
(D(I)-D)([L]-[L'])=\sum [L_j]-\sum [L'_j]=-([L]-[L']) +([L]+\sum
[L_j]) - ([L']+\sum [L'_j])
$$
We also know that $([L]+\sum [L_j]) - ([L']+\sum [L'_j])$ is the
pull-back of $x-x'$ on $C$, where $x=L\cap C$ and $x'=L'\cap C$.
Since $C$ has rational components, the class of $x-x'$ is trivial.
Hence we get
$$
(D(I)-D)([L]-[L'])=-([L]-[L']).
$$
This implies that $i$ is induced by $D$. By definition, $D$ is
symmetric and effective. We still need to show that $D$ is fixed
point free. If $[L]\in\tilde{C}$ is a fixed point of $D$. Let $x$ be
the point of $L$ meeting $C$. This is well defined even if $L$ is a
secant line of $C$. This is because when $L$ a secant line of $C$,
the point $[L]$ of $f^{-1}C$ is a nodal point, here $f$ is the
morphism in \eqref{universal line cubic threefold}. Hence $[L]$
gives rise to two points $[L]$ and $[L']$ of the normalization
$\tilde{C}$ corresponding to the choice of a point of $L\cap C$. Let
$\Pi$ be the plane spanned by the tangent direction of $C$ at $x$
and the line $L$. The condition $[L]\in D([L])$ implies that
$\Pi\cdot X=2L+L'$. This implies that $L\subset B$ and $C$ is
tangent to $B$ at the point $x$. But this is not allowed since $C$
is admissible. The statement (c) follows from (b) by applying
\cite{kanev} (Theorem 3.1.). Furthermore, (c) implies that
$\ker(i-1)$ is connected and equal to $\mathrm{im}(i+q-1)$, see
Lemma 7.7 of \cite{bloch-murre}. Hence in diagram \eqref{diagram
prym-tjurin}, the morphism $\eta_*$ factors through
$\mathrm{Pr}(\tilde{C},i)$ via $i_0$. In this way we get
$u'_C:\mathrm{Pr}(\tilde{C},i)\rightarrow \Alb(F)$ such that
$u'_C\circ i_0=\eta_*$. We take $u_C=\psi_F\circ u'_C$. Hence we
only need to prove that $u'_C$ induces an isomorphism between
principally polarized abelian varieties. To do this, let
$\psi:\tilde{C}\rightarrow A=\Alb(F)$ be the natural morphism
(determined up to a choice of a point on $\tilde{C}$). After
identifying $A$ with $J(X)$, we have
\begin{equation}\label{condition i of Masiewicki}
\psi(D([L]))= \sum \psi([L_i])=\sum
L_i=(1-2e)L-2C+\text{const.}=(1-q)\psi([L])+\text{const.}
\end{equation}
Also notice that $\eta_*$ is epimorphism (since
$\eta(\tilde{C})$ is ample on $F$). By a theorem of Welters (see
Theorem 5.4 of \cite{kanev}), this condition together with
\eqref{condition b of kanev} implies
\begin{equation}\label{condition ii of Masiewicki}
\psi_*(\tilde{C})\equiv (q/24)\theta^4
\end{equation}
where $\theta$ is a theta divisor of $A$. Then by Theorem 5.6 of
\cite{kanev}, we know that $(A,\theta)$ is a direct summand of
$(\mathrm{Pr}(\tilde{C},i),\Xi)$. But we know that
$\eta^*\circ\lambda_I$ is surjective, hence $u'_C$ has to be an
isomorphism.
\end{proof}

Now let $C_1$ and $C_2$ be two smooth admissible rational curves on
$X$ which do not meet. Let $C=C_1\cup C_2$ be the disjoint union of
two admissible curves and assume that $C$ is again admissible. We
use the a subscription for notations of the corresponding curve. So
we have $\tilde{C}_1$, $\tilde{C}_2$, $\tilde{C}=\tilde{C}_1\cup
\tilde{C}_2$ and correspondingly $\tilde{J}_1$, $\tilde{J}_2$,
$\tilde{J}=\tilde{J}_1\oplus \tilde{J}_2$. Assume that $\ch(k)\nmid
q_1q_2$. The correspondence $D$ on $\tilde{C}$ can be written as
$$
D=\begin{pmatrix} D_{1} &D_{12}\\
  D_{21} &D_{2} \end{pmatrix}
$$
Let $[L]\in \tilde{C}_1$, then $D_{21}([L])=\sum [L_j]$ where $L_j$
are all secant lines of $(C_2,L)$. We also see that
$^tD_{12}=D_{21}$.

\begin{prop}
Let notations and assumptions be as above. Then the following hold\\
(a) The induced morphism $D_{21}:\tilde{J}_1\rightarrow \tilde{J}_2$
fits into the following commutative diagram
$$
\xymatrix{
 \tilde{J}_1\ar[r]^{D_{21}}\ar[d]_{i_1+q-1} &\tilde{J}_2
 \ar[d]^{i_2-1}\\
 \tilde{J}_1\ar[r]^{-D_{21}} &\tilde{J}_2
}
$$
(b) The following equalities hold
$$
 D_{21}(i_1+q_1-1)=(i_2+q_2-1)D_{21}=0
$$
Also we have
$\mathrm{Im}(D_{21})=\mathrm{Pr}(\tilde{C}_2,i_2)\subset\tilde{J}_2$
and the endomorphism $D_{21}$ factors through
$i_1-1:\tilde{J}_1\rightarrow\mathrm{Pr}(\tilde{C}_1,i_1)$ and
induces an isomorphism
$$
t_{21}:(\mathrm{Pr}(\tilde{C}_1,i_1),\Xi_1)\rightarrow
(\mathrm{Pr}(\tilde{C}_2,i_2),\Xi_2)
$$
of principally polarized abelian varieties.\\
(c) The isomorphism $t_{21}$ is compatible with $u_{C_1}$ and
$u_{C_2}$, namely $u_{C_2}\circ t_{21}=u_{C_1}$.
\end{prop}

\begin{proof}
By (a) of Theorem \ref{J(X) as prym-tjurin}, we know that
$(D-1)(D+q-1)=0$ as an endomorphism of $\tilde{J}$. This can be
written in the following matrix form,
\begin{equation}\label{matrix equation}
\begin{pmatrix}
q_2(i_1-1)+D_{12}D_{21} &(i_1-1)D_{12}+D_{12}(i_2+q-1)\\
D_{21}(i_1+q-1)+(i_2-1)D_{21} &D_{21}D_{12}+q_1(i_2-1)
\end{pmatrix}
=0
\end{equation}
In particular, we have $D_{21}(i_1+q-1)+(i_2-1)D_{21}=0$. This is
the same as the commutativity of the diagram in (a). If we write
$P_1=\mathrm{Pr}(\tilde{C}_1,i_1)$ and
$Q_1=\mathrm{Im}(i_1+q-1)\subset \tilde{J}_1$. Then
$\tilde{J}_1=P_1+Q_1$ and $i_1|_{P_1}=1-q_1$, $i_1|_{Q_1}=1$, see
\cite{bloch-murre}. Hence the morphism $i_1+q-1$ is surjective. This
implies that $\mathrm{Im}(D_{21})\subset
P_2=\mathrm{Pr}(\tilde{C}_2,i_2)$, which implies
$(i_2+q_2-1)D_{21}=0$. Hence we get the following commutative
diagram
$$
\xymatrix{
 \tilde{J}_1\ar[r]^{D_{21}}\ar[d]_{i_1+q-1} &P_2
 \ar[d]^{-q_2}\\
 \tilde{J}_1\ar[r]^{-D_{21}} &P_2
}
$$
Thus $q_2D_{21}(\alpha)=D_{21}(i_1+q-1)(\alpha)=D_{21}(q\alpha)$,
i.e. $q_1D_{21}(\alpha)=0$, for all $\alpha\in Q_1$. This implies
that $D_{21}(i_1+q_1-1)=0$. Hence $D_{21}$ factors through $P_1$ and
gives the morphism $t_{21}:P_1\rightarrow P_2$, i.e.
$t_{21}\circ(i_1-1)=D_{21}$. One easily verifies that
$D_{21}|_{P_1}=-q_1t_{21}$. We have similar identities for $D_{12}$
and $t_{12}$. Note that the equations of the diagonal entries of
\eqref{matrix equation} lead to the following
$$
(D_{12}D_{21})|_{P_1}=q_1q_2,\quad (D_{21}D_{12})|_{P_2}=q_1q_2
$$
These relations give $t_{21}t_{12}=1$ and $t_{12}t_{21}=1$. The
compatibility of $t_{21}$ and $u_i=u_{C_i}$ is an easy diagram
chasing from
$$
\xymatrix{
 \tilde{J}_1\ar[d]_{i_1-1} \ar[r]^{D_{21}} &P_2\ar[d]^{u_2}\ar[r]^{j_2} &\tilde{J}_2\ar[d]^{\psi_{\tilde{J}_2}}\\
 P_1\ar[r]^{u_1}\ar[ru]_{t_{21}} &J(X)\ar[r]^{-q_2} &J(X)
}
$$
Note that $u_1\circ(i_1-1)=\psi_{\tilde{J}_1}$ and the up left
triangle, the right square and the outside square are all
commutative. Then one easily sees that $q_2(u_2\circ t_{21}-u_1)=0$,
which implies $u_2\circ t_{21}=u_1$.
\end{proof}

\begin{cor}
Let $C\subset X$ be a smooth admissible rational curve and
$\ch(k)\nmid q$. Let $L$ and $L'$ be two lines without meeting $C$.
Let $L_i$ (resp. $L'_i$) be all the secant lines of the pair $(L,C)$
(resp. $(L',C)$). Then
$$
\sum [L_i]-\sum [L'_i]\in \mathrm{Pr}(\tilde{C},i)\subset \tilde{J}
$$
\end{cor}

\begin{proof}
Choose another smooth admissible rational curve $C_1$ that meets
both $L$ and $L'$ transversally in a single point. Take $C_2=C$,
then the left hand side is just $D_{21}([L]-[L'])$.
\end{proof}

Fix a smooth admissible rational curve $C\subset X$ such that
$\ch(k)\nmid q$. The above corollary allows us to define a map
$$
v'_C:A_1(X)\rightarrow \mathrm{Pr}(\tilde{C},i),\quad L-L'\mapsto
\sum [L_i]-\sum [L'_i]
$$

\begin{cor}
The above map induces an morphism $v_C:J(X)\rightarrow
\mathrm{Pr}(\tilde{C},i)=P$, which is the same as $u^{-1}_{C}$.
\end{cor}
\begin{proof}
We have the following commutative diagram
$$
\xymatrix{
 J(X)\ar[r]^{v_C}\ar[d]_{-q} &P\ar[rd]_{-q}\ar@{^(->}[r]^{j}&\tilde{J}\ar[d]\ar[r]^{\psi_{\tilde{J}}\quad} &J(X)\\
 J(X)\ar[rr]^{v_C}&&P\ar[ru]_{u_C} &
}
$$
Then we have $-qu_C\circ v_C=\psi_{\tilde{J}}\circ j\circ v_C=-q$ on
$J(X)$. This implies that $u_C\circ v_C=1$.
\end{proof}

Now we study the Abel-Jacobi map associate to a family of curves on
$X$. Let $\mathscr{C}\subset (X\times T)$ be a family of curves on
$X$ parameterized by $T$. The Abel-Jacobi map $\psi_T$ associated to
this family is given by sending $t$ to the class of $\mathscr{C}_t$.
Assume that for a general point $t\in T$, the curve $\mathscr{C}_t$
does not meet $C$ and the pair $(\mathscr{C}_t,C)$ has finitely many
(generalized) secant lines. The rule associating all the
(generalized) secant lines of $(\mathscr{C}_t,C)$ to the point $t$
defines a correspondence
$$
\Psi_{T,C}:T\dashrightarrow \tilde{C}
$$
Up to the choice of a base point of $T$, one gets a morphism
$$
\Psi_{T,\tilde{J}}:T\rightarrow \tilde{J}
$$

\begin{prop}
Let $T$ and $\Psi_{T,\tilde{J}}$ be as above, then the following are
true\\
(a) The image of $\Psi_{T,\tilde{J}}$ is in
$P=\mathrm{Pr}(\tilde{C},i)\subset\tilde{J}$.\\
(b) The composition $\xymatrix{T\ar[r]^{\Psi_{T,\tilde{J}}}
&P\ar[r]^{u_C} &J(X)}$ is identified with the Abel-Jacobi map
$\psi_T$.
\end{prop}

\begin{proof}
By definition, $\Psi_{T,\tilde{J}}(t)=v'_C(\mathscr{C}_t)$. Then the
proposition follows easily.
\end{proof}

\end{document}